\documentclass[reqno]{amsart}
\newcommand \datum {February 28, 2021}
\newcommand \dateone {February 28, 2021}

\usepackage{amssymb,latexsym}
\usepackage{amsmath}
\usepackage{amscd}
\usepackage{graphicx}
 \usepackage{color}
\usepackage{enumerate}
\numberwithin{equation}{section}
\theoremstyle{plain}
 \newtheorem{theorem}{Theorem}[section]
 \newtheorem{lemma}[theorem]{Lemma}
 \newtheorem{corollary}[theorem]{Corollary}
\theoremstyle{definition}
 \newtheorem{definition}[theorem]{Definition}
 \newtheorem{convention}[theorem]{Convention}
 \newtheorem{remark}[theorem]{Remark}
 \newtheorem*{addedone}{Added on \dateone} 

\newenvironment{enumeratei}{\begin{enumerate}[\quad\upshape (i)]} {\end{enumerate}}
\newcommand \tbf[1] {\textbf{#1}}  

\newcommand \Jir [1] {\textup J(#1)} 
\newcommand \Mir [1] {\textup{M}(#1)}
\newcommand \Nplu {\mathbb N^+}
\newcommand \Rplu {\mathbb R^+}
\renewcommand \phi{\varphi}
\newcommand \restrict [2] {{#1}\kern-1pt \rceil_{\kern-1pt #2}}
\DeclareMathOperator{\Con}{Con}
\newcommand \brho {{\boldsymbol{\rho}}}
\newcommand \bsigma {{\boldsymbol{\sigma}}}
\newcommand \bTheta {{\boldsymbol{\Theta}}}
\newcommand \btau {{\boldsymbol{\tau}}}

\newcommand \leftb [1]  {\textup{C}_{\textup{left}}(#1)} 
\newcommand \rightb [1] {\textup{C}_{\textup{right}}(#1)} 
\newcommand \ideal[1]{\mathord\downarrow #1}
\newcommand \filter[1]{\mathord\uparrow #1}

\newcommand \Edges[1] {\textup{Edge}(#1)} 
\newcommand \rellambda {\mathrel{\lambda}}
\newcommand \bdia {\mathcal C_1}
\newcommand \cdia {\mathcal C_2}

\newcommand \con {\textup{con}}
\newcommand \intv [1]{{\mathfrak #1}}  
\newcommand \inp {{\mathfrak p}}
\newcommand \inq {{\mathfrak q}} 
\newcommand \set [1]{\{#1\}}
\newcommand \tuple [1] {\langle #1 \rangle}
\newcommand \pair [2] {\tuple{#1,#2}}
\newcommand \cornl [1] { \textup{lc}(#1) }
\newcommand \cornr [1] { \textup{rc}(#1) } 

\newcommand \Sn[1] {S_7^{(#1)}}
\newcommand \foot [1] {\textup{Foot}(#1)}
\newcommand \Bot[1] {\beta_{#1}}
\newcommand \Enl [1] {\textup{Lit}(#1)}
\newcommand \pEnl [1] {\textup{Lit}^+(#1)}
\newcommand \LEnl [1] {\textup{LeftLit}(#1)}
\newcommand \REnl [1] {\textup{RightLit}(#1)}
\newcommand \Roof [1] {\textup{Roof}(#1)}
\newcommand \Floor [1] {\textup{Floor}(#1)}
\newcommand \LRoof [1] {\textup{LRoof}(#1)}
\newcommand \LFloor [1] {\textup{LFloor}(#1)}
\newcommand \RRoof [1] {\textup{RRoof}(#1)}
\newcommand \RFloor [1] {\textup{RFloor}(#1)}
\newcommand \ashape {$\pmb{\pmb{\wedge}}$}
\newcommand \Lamps[1] {\textup{Lamp}(#1)}
\newcommand \rhgeomb {\brho_{\textup{Body}}}
\newcommand \lrhgeomb {\brho_{\textup{LRBody}}}

\newcommand \rhgeomc {\brho_{\textup{CircR}}}

\newcommand \lrhgeomc {\brho_{\textup{LRCircR}}}
\newcommand \rhfoot {\brho_{\textup{foot}}}
\newcommand \rhinfoot {\brho_{\textup{infoot}}}
\newcommand \rhinpfoot {\brho_{\textup{in+foot}}}
\newcommand \rhalg {\brho_{\textup{alg}}}
\newcommand \cirrec [1] {\textup{CircR}(#1)}
\newcommand \body [1] {\textup{Body}(#1)}
\newcommand \topedge [1] {\textup{TopE}(#1)}
\newcommand \Trajs [1] {\textup{Traj}(#1)}
\newcommand \jirdprec {\mathrel{\prec_{\kern-1pt\Jir D}}}
\newcommand \Max [1] {\textup{Max}(#1)}
\newcommand \lmax[1] {\textup{LeftMax}(#1)}
\newcommand \rmax[1] {\textup{RightMax}(#1)}
\newcommand \Her [1] {\textup{DnSt}(#1)}
\newcommand \lift [1] {\textup{lifted}(#1)}
\newcommand \Peak [1] {\textup{Peak}(#1)}
\newcommand \cov [1]  {#1^+}
\newcommand \Lmp [1]  {\textup{Lmp}(#1)}
\newcommand \lprec  {\mathrel{\prec_{\textup{left}}}}
\newcommand \rprec  {\mathrel{\prec_{\textup{right}}}}
\newcommand \defiff {\overset{\textup{def}}{\iff}}

\newcommand \idez [1] {``#1''}

\newcommand\red[1]{{\textcolor{red}{#1}}}

\newcommand \nothing [1] {}

\begin{document}
\title[Lamps in slim rectangular lattices]
{Lamps in slim rectangular planar semimodular lattices}

\author[G.\ Cz\'edli]{G\'abor Cz\'edli}
\email{czedli@math.u-szeged.hu}
\urladdr{http://www.math.u-szeged.hu/~czedli/}
\address{University of Szeged, Bolyai Institute. 
Szeged, Aradi v\'ertan\'uk tere 1, HUNGARY 6720}

\begin{abstract} A planar (upper) semimodular lattice $L$ is \emph{slim} if the five-element nondistributive modular lattice $M_3$ does not occur among its sublattices. (Planar lattices are finite by definition.) 
\emph{Slim rectangular lattices} as particular slim planar semimodular lattices  were defined
by G.\ Gr\"atzer and E.\ Knapp in 2007.  In 2009, they also  proved that the congruence lattices of slim planar semimodular lattices with at least three elements are the same as those of slim rectangular lattices.
In order to provide an effective tool for studying these congruence lattices, we introduce the concept of \emph{lamps} of slim rectangular lattices and prove several of their properties. 
Lamps and several tools based on them allow us to prove in a new and easy way that the congruence lattices of slim planar semimodular lattices satisfy the two previously known properties.  
Also, we use  lamps to prove that these congruence lattices satisfy four new properties including the \emph{two-pendant four-crown property} and the \emph{forbidden marriage property}.
\end{abstract}

\thanks{This research was supported by the National Research, Development and Innovation Fund of Hungary under funding scheme K 134851.}


\dedicatory{Dedicated to L\'aszl\'o K\'erchy on his seventieth birthday}


\date{\datum\hfill{\red{Hint: check the author's website for possible updates}}}

\maketitle

\subsection*{Note on the dedication}
Professor \emph{L\'aszl\'o K\'erchy} is the previous  editor-in-chief of Acta Sci.\ Math.\ (Szeged). 
I have been knowing him since 1967, when I enrolled in a high school where he was a  second-year student with widely acknowledged mathematical talent.  His influence had played a  role that I became a student at the Bolyai (Mathematical) Institute in Szeged. Furthermore, he helped me to become one of his roommates in  Lor\'and E\"otv\"os University Dormitory in 1972. Leading by his example, I could become a member of the Bolyai Institute right after my graduation. With my gratitude, I dedicate this paper to his seventieth birthday.

\section{Introduction} 
The theory of planar semimodular lattices has been an intensively studied part of lattice theory since Gr\"atzer and Knapp's pioneering \cite{gratzerknapp1}. The key role in the theory of these lattices is played by \emph{slim} planar semimodular lattices; their definition is postponed to Section~\ref{sectionlamps}.
%
%
The importance of slim planar semimodular lattices is surveyed,  for example, in Cz\'edli and Kurusa~\cite{czgkurusa}, Cz\'edli and Gr\"atzer~\cite{czgggltsta}, Cz\'edli and Schmidt~\cite{czgschtJH}, and
Gr\"atzer and Nation~\cite{gr-nation}. \emph{Slim rectangular lattices}, to be defined later, were introduced in Gr\"atzer and Knapp~\cite{gratzerknapp3}  as  particular slim planar semimodular lattices and they will play a crucial role in our proofs.  
The study of \emph{congruence lattices} $\Con L$ of slim planar semimodular lattices $L$ goes back to Gr\"atzer and Knapp~\cite{gratzerknapp3}. 
These congruence lattices are finite distributive lattices and, in addition, we know from Cz\'edli~\cite{czganotesps} and   Gr\"atzer~\cite{ggonaresczg}, \cite{ggconprimperst}, and \cite{ggSPS8}  that they have special properties.

\subsection*{Target} One of our targets is to develop effective tools to derive the above-mentioned special properties in a new and easy way and to present four new properties. To do so, we are going to define the  \emph{lamps} of a slim rectangular lattice $L$ so that the set of lamps becomes a poset (partially ordered set) isomorphic to the poset $\Jir {\Con L}$ of join-irreducible congruences of $L$. It remains a problem whether the properties recalled or proved in the present paper and in Cz\'edli and Gr\"atzer~\cite{czgginprepar} characterize the congruence lattices of slim planar semimodular lattices.

\subsection*{Outline} For comparison with our lamps, 
the rest of the present section mentions the known ways to describe $\Jir {\Con L}$ for a finite lattice $L$.  In Section~\ref{sectionlamps}, we 
recall the concept of slim planar semimodular lattices, that of   slim rectangular lattices, and that of their $\bdia$-diagrams. Also in Section~\ref{sectionlamps}, we introduce the concept of lamps of these lattices and prove our (Main) Lemma~\ref{lemmamain}. This lemma provides the main tool for \idez{illuminating} the congruence lattices of slim planar semimodular lattices. 
In Section~\ref{sectioneasycons}, further tools are given and several  consequences of (the Main) Lemma~\ref{lemmamain} are proved. In particular, this section proves that  the congruence lattices of slim planar semimodular lattices satisfy  both previously known properties, see Corollaries~\ref{corollstzRsG} and \ref{coroltwocover}, and two new properties, see Corollaries~\ref{corolnWdWr} and \ref{coroldioec}.
Section~\ref{sectionfourcrown} defines the two-pendant four-crown property and the forbidden marriage property, and proves that the congruence lattices of  slim planar semimodular lattices satisfy these two properties, too; see Theorem~\ref{thmmain}.

\subsection*{Comparison with earlier approaches to $\Jir {\Con L}$}
Let $L$ be a finite lattice; its \emph{congruence lattice} is denoted by $\Con L$.   Since $\Con L$ is distributive, it is determined by the poset  $\Jir{\Con L}=\tuple{\Jir{\Con L};\leq}$ of its nonzero \emph{join-irreducible elements}. There are three known ways to describe this poset.
 
First, one can use the \emph{join dependency relation} defined on $J(L)$; see Lemma 2.36 of the monograph Freese, Je\v zek, and Nation~\cite{fjnbook}, where this relation is attributed to  Day~\cite{day}. 

Second, Gr\"atzer~\cite{ggconprimperst} takes (and well describes) the prime-perspectivity relation on the set of \emph{prime intervals} of $L$. His description becomes more powerful if $L$ happens to be a slim planar semimodular lattice: for such a lattice, Gr\"atzer's Swing Lemma applies, see  \cite{ggswinglemma} and see also
Cz\'edli, Gr\"atzer, and Lakser~\cite{czggghlswing} and Cz\'edli and Makay~\cite{czgmakay}. 

Third, but only for a slim rectangular lattice $L$, Cz\'edli~\cite{czgtrajcolor} defined a relation on the set of \emph{trajectories} of $L$ while Gr\"atzer~\cite{ggswinglemma} defined an analogous relation on the set of prime intervals.  Although it happened in a different way, 
Theorem 7.3(ii) of Cz\'edli~\cite{czgtrajcolor} indicates that the dual of the approach based on join dependency relation could also have been used 
to derive a more or less similar description of $\Jir{\Con L}$.

A relation $\rho\subseteq X^2$ on a set $X$ is a \emph{quasiorder} (also called \emph{preorder}) if it is reflexive and transitive. 
Each of the three approaches mentioned above defines only a quasiorder on a set $X$ in the first step; this set consists of 
join-irreducible elements, prime intervals, or trajectories, respectively. In the next step, we have to form 
the quotient set $X/(\rho \cap \rho^{-1})$ and equip it with the 
quotient relation  $\rho/(\rho\cap\rho^{-1})$ to obtain $\tuple{\Jir{\Con L};\leq}$ up to isomorphism. For a slim rectangular lattice $L$,  our lamps provide a more efficient description of $\Jir{\Con L}$ since we do not have to form a quotient set.

\section{From diagrams to lamps}\label{sectionlamps}
By definition, planar lattices are finite. A \emph{slim planar semimodular lattice} is a planar (upper) semimodular lattice $L$ such that one of the following three conditions holds:
\begin{enumeratei}
\item{} $M_3$, the five-element nondistributive modular lattice, is not
a sublattice of $L$,
\item $M_3$ is not a cover-preserving sublattice of $L$,
\item $\Jir L$, the \emph{set of nonzero join-irreducible elements} of $L$, is the union of two chains;
\end{enumeratei}
see Gr\"atzer and Knapp~\cite{gratzerknapp1} and Cz\'edli and Schmidt~\cite{czgschtJH}, or the book chapter Cz\'edli and Gr\"atzer~\cite{czgggltsta}
 for the equivalence of these three conditions for planar semimodular lattices (but not for other lattices.)

Let $L$ be a  slim planar semimodular  lattice; we \emph{always} assume that a planar diagram of $L$ is fixed. The \emph{left boundary chain} and the \emph{right boundary chain} of $L$ are denoted by 
$\leftb L$ and $\rightb L$, respectively. Here and at several other concepts occurring later, we heavily rely on the fact that the diagram of $L$ is fixed; indeed, $\leftb L$ and $\rightb L$ depend on the diagram, not only on $L$. 

Following Gr\"atzer and Knapp~\cite{gratzerknapp3}, a slim planar semimodular lattice is called a \emph{slim rectangular lattice}
if $|L|\geq 4$, $\leftb L$ has exactly one doubly irreducible element, $\cornl L$, 
 $\rightb L$ has exactly one doubly irreducible element, $\cornr L$, and these two doubly irreducibly elements are complementary, that is, $\cornl L\vee\cornr L=1$ and $\cornl L\wedge \cornr L=0$. 
Here $\cornl L$ and $\cornr L$ are called the \emph{left corner} (element) and the \emph{right corner} (element) of the rectangular lattice $L$. Note that  $|L|\geq 4$ above can be replaced by  $|L|\geq 3$. Note also that the definition of rectangularity does not depend on how the diagram is fixed since  $\cornl L$ and $\cornr L$ are the only doubly irreducible elements of $L$. 

Let us emphasize that a slim rectangular lattice is planar and semimodular by definition, whereby the title of the paper is redundant. The purpose of this redundancy is to  give more information about the content of the paper.

The (principal) ideals $\ideal \cornl L$ and $\ideal \cornr L$ are chains and they are called the \emph{bottom left boundary chain} and the \emph{bottom right boundary chain}, respectively, while the filters $\filter \cornl L$ and $\filter \cornr L$ are also chains, the \emph{top left boundary chain} and  \emph{top right boundary chain}, respectively. 
The \emph{lower boundary} and the \emph{upper boundary} of $L$ are $\ideal \cornl L\cup\ideal \cornr L$ and $\filter \cornl L\cup\filter \cornr L$, respectively. 
Note also that $\Jir L\cup\set 0$ equals the lower boundary $\ideal\cornl L\cup \ideal\cornr L$,  but for the set $\Mir L$ of non-unit \emph{meet-irreducible elements}, we only have that $\filter\cornl L\cup\filter\cornr L\subseteq \Mir L\cup\set 1$. For example, the lattices $\Sn n$ for $n\in\Nplu:=\set{1,2,3,\dots}$, defined in 
 Cz\'edli~\cite{czgtrajcolor} and presented here in Figure~\ref{figsn} for $n\leq 4$, are slim rectangular lattices. 

\begin{figure}[htb]
\centerline
{\includegraphics[width=\textwidth]{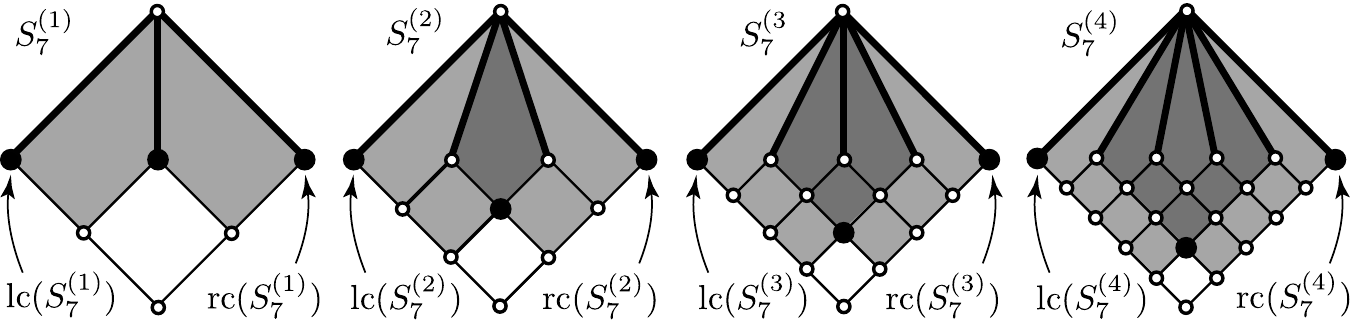}}      
\caption{$\Sn 1$, $\Sn 2$, $\Sn 3$, and  $\Sn 4$}\label{figsn}
\end{figure}

If $p$ and $q$ are elements of a lattice such that $p\prec q$, then the \emph{prime interval}, that is, the two-element interval $[p,q]$ is an \emph{edge} of the diagram. 
Following Cz\'edli~\cite{czgrectectdiag}, we need the following concepts. 

\begin{definition}[Types of diagrams] The \emph{slope} of the line $\set{\pair x x: x\in\mathbb R}$ and that of the line $\set{\pair x {-x}: x\in\mathbb R}$ are called \emph{normal slopes}. This allows us to speak of lines, line segments, and edges of normal slopes. For example, an edge $[p, q]$ of a lattice diagram is of a normal slope iff the angle that this edge makes with a horizontal line is $\pi/4$ ($45^\circ$) or  $3\pi/4$ ($135^\circ$).   If this angle is strictly between $\pi/4$ and  $3\pi/4$, then the edge is \emph{precipitous}. For examples, vertical edges are precipitous. 
We say that a diagram of a slim rectangular lattice $L$ belongs to $\bdia$ or, in other words, it is a \emph{$\bdia$-diagram} if every edge $[p, q]$ such that $p\in \Mir L\setminus(\leftb L\cup\rightb L)$ is precipitous and all the other edges are of normal slopes. A $\bdia$-diagram of $L$ belongs to $\cdia$ or, shortly saying, it is a \emph{$\cdia$-diagram} if any two edges on the lower boundary are of the same geometric length. 
\end{definition}

The diagrams in Figures~\ref{figsn}, \ref{figslns}, \ref{figftnt}, and \ref{figabcd}, $L_6=L$ in Figure~\ref{figmlsrNgrt}, and 
the diagrams denoted by $L$ in Figures~\ref{figlot} and \ref{figleri} are $\cdia$-diagrams.  In addition to these $\cdia$-diagrams, the diagrams in Figures~\ref{figmfXtns} and \ref{figmlsrNgrt} are $\bdia$-diagrams while the diagram in Figure~\ref{figfbmr} can be a part of a $\bdia$-diagram. In fact, 
\emph{all} diagrams of slim rectangular lattices in this paper are $\bdia$-diagrams. Note that we believe  that only $\bdia$-diagrams can give satisfactory insight into the congruence lattices of slim rectangular lattices. 
For diagrams, drawn or not,  let us agree in the following.

\begin{convention}\label{convxpllzT} In the rest of the paper, all diagrams of slim rectangular lattices are assumed to be $\bdia$-diagrams.   Furthermore, $L$ will always denote a slim rectangular lattice with a fixed $\bdia$-diagram. 
\end{convention}

The reader may wonder how the new concepts in the following definition obtained their name; the explanation will be given in the paragraph preceding Definition~\ref{defenlstS}.

\begin{definition}[Lamps] Let $L$ be a slim rectangular lattice with a fixed $\bdia$-diagram.
\begin{enumeratei}
\item
An edge $\intv n=[p,  q]$ of   $L$ is a \emph{neon tube} if $p\in \Mir L$. The elements $p$ and $q$ are the \emph{foot}, denoted by $\foot {\intv n}$, and the \emph{top} of this neon tube. Clearly, a neon tube is determined by its foot. 
\item If an edge  $[p, q]$ is a neon tube such that $p$ belongs to the boundary of $L$ (equivalently, if $\ideal p$ contains a doubly irreducible element),
then  $[p, q]$ is also a \emph{boundary lamp} (with a unique neon tube $[p,q]$).   If $I=[p,q]$ is a boundary lamp, then $p$ is called the \emph{foot} of $I$ and is denoted by $\foot I$ while $\Peak I:=q$ is the \emph{peak} of $I$. Note the terminological difference: neon tubes have tops but lamps have peaks.
\item
Assume that $q\in L$ is the top of a neon tube whose foot is not on the boundary $\leftb L\cup\rightb L$ of $L$, and let 
\begin{equation}
\Bot q:=\bigwedge\set{p_i: [p_i, q]\text{ is a neon tube and $p_i\notin\leftb L\cup\rightb L$}}.
\label{eqdkHtFgWgwLNm}
\end{equation}
Then the interval $I:=[\Bot q,q]$ is an \emph{internal lamp} of $L$.  The prime intervals $[p_i,  q]$ such that $p_i\in \Mir L$ 
but  $p_i\notin\leftb L\cup\rightb L$ are the \emph{neon tubes} of this lamp. If $I$ is an internal lamp, then either $I$ is a neon tube and we say that $I$ has a unique neon tube, or $I$ has more than one neon tubes. The element $q$ is the \emph{peak} of the lamp $I$ and it is denoted by $\Peak I$ while $\foot I:=\Bot q$ is the \emph{foot} of $I$. 
\item The \emph{lamps} of $L$ are its boundary lamps and its internal lamps. Clearly, for every lamp $I$ of $L$, 
\begin{equation}
\foot I =\bigwedge\set{\foot {\intv n}: \intv n\text{ is a neon tube of } I}. 
\label{eqdRsrFsPd}
\end{equation}
\end{enumeratei}
\end{definition}

Since a slim rectangular lattice $L$ has only two doubly irreducible elements, $\cornl L$ and $\cornr L$, the non-containment in \eqref{eqdkHtFgWgwLNm} is equivalent to the condition that ``$\ideal p_i$ does not contain a doubly irreducible element''. Therefore, the concept of lamps does not depend on the diagram of $L$. Note that in a reasonable sense, the $\bdia$ diagram of $L$ is unique, and so is its $\cdia$ diagram; see Cz\'edli~\cite{czgrectectdiag}. To help the reader in finding the lamps in our diagrams, let us agree to the following.

\begin{convention}\label{convsRhRtskTskk}
In our diagrams  of slim rectangular lattices, the foots of lamps are exactly the black-filled elements. (Except possibly for Figure~\ref{figfbmr}, which can be but need not be the whole lattice in question.) The thick edges  are always neon tubes (but there can be neon tubes that are not thick edges).
\end{convention}

Note that in (the slim rectangular lattices of) Figures~\ref{figsn}, \ref{figlot}, \ref{figleri}, and \ref{figabcd}, the neon tubes are exactly the thick edges. 
In addition to the fact that neon tubes are easy to recognize as edges with bottom elements in $\Mir L$, there is another way to recognize them even more easily; the following remark  follows from definitions.

\begin{remark}\label{remarksZsSTmVbQxPT}
 Neon tubes in a $\bdia$-diagram of a slim rectangular lattice (and so in our figures)  are exactly the precipitous edges and the edges on the upper boundary $\filter \cornl L\cup\filter \cornr L$.
\end{remark}

\emph{Regions} of a slim rectangular lattice $L$ are defined as closed planar polygons surrounded by some edges of (the fixed diagram) of $L$; see Kelly and Rival~\cite{kellyrival} for an elaborate treatise of these regions, or see Cz\'edli and Gr\"atzer~\cite{czgggltsta}. Note that every interval of a planar lattice determines a region; possibly of area 0 if the interval is a chain.
The affine plane on which diagrams are drawn is often identified with $\mathbb R^2$ via the classical coordinatization. 
By the \emph{full geometric rectangle} of a slim rectangular lattice $L$ with a fixed diagram we mean the closed geometric  rectangle whose boundary is the union of all edges belonging to $\leftb L\cup\rightb L$. It is a rectangle indeed since we allow $\bdia$ diagrams only. 
Smaller geometric rectangles are also relevant; this is why we have the second part of the following definition.  

\begin{definition}[geometric shapes associated with lamps]\label{defgshscdzrkrrl}
 Keeping Convention~\ref{convxpllzT} in mind, 
let $I=[p,q]=[\foot I,\Peak I]$ be a lamp of $L$. 
\begin{enumeratei}
\item\label{defgshscdzrkrrla} The \emph{body} of $I$, denoted by $\body I=\body{[p,q]}$ is the region determined by $[p,q]$. Note that $\body I$ is a line segment if $I$ has only one neon tube, and (by Remark~\ref{remarksZsSTmVbQxPT}) it is a quadrangle of positive area having two precipitous upper edges and two lower edges of normal slopes otherwise. 
\item\label{defgshscdzrkrrlb} Assume that $I$ is an internal lamp, and define $r$ as the meet of all lower covers of $q$. Then the interval $[r,q]$ is a region; this region is denoted by $\cirrec I=\cirrec{[p,q]}$ and it is called the \emph{circumscribed rectangle} of $I$.
\end{enumeratei} 
\end{definition}

For example, if $I$ is the only internal lamp of $\Sn n$, then $\cirrec I$ is the full geometric rectangle of $\Sn n$ for all $n\in \Nplu$ while $\body I$ is the dark-grey area for $n\in\set{2,3,4}$ in Figure~\ref{figsn}. To see another example, if $E$ is the lamp with two neon tubes labelled by $e$ on the left of  Figure~\ref{figlot}, then $\body I$ is the dark-grey area and the vertices of $\cirrec I$ are $x$, $y$, $z$, and $t$.
Note that by the dual of Cz\'edli~\cite[Proposition 3.13]{czgcircles}, 
\begin{equation}
\parbox{9cm}{$r$ in Definition~\ref{defgshscdzrkrrl}\eqref{defgshscdzrkrrlb} can also be defined as the meet of the leftmost lower cover and the rightmost lower cover of $q$.}
\label{eqpbxszhvgLtgTsnK}
\end{equation}

\begin{figure}[htb]
\centerline
{\includegraphics[width=\textwidth]{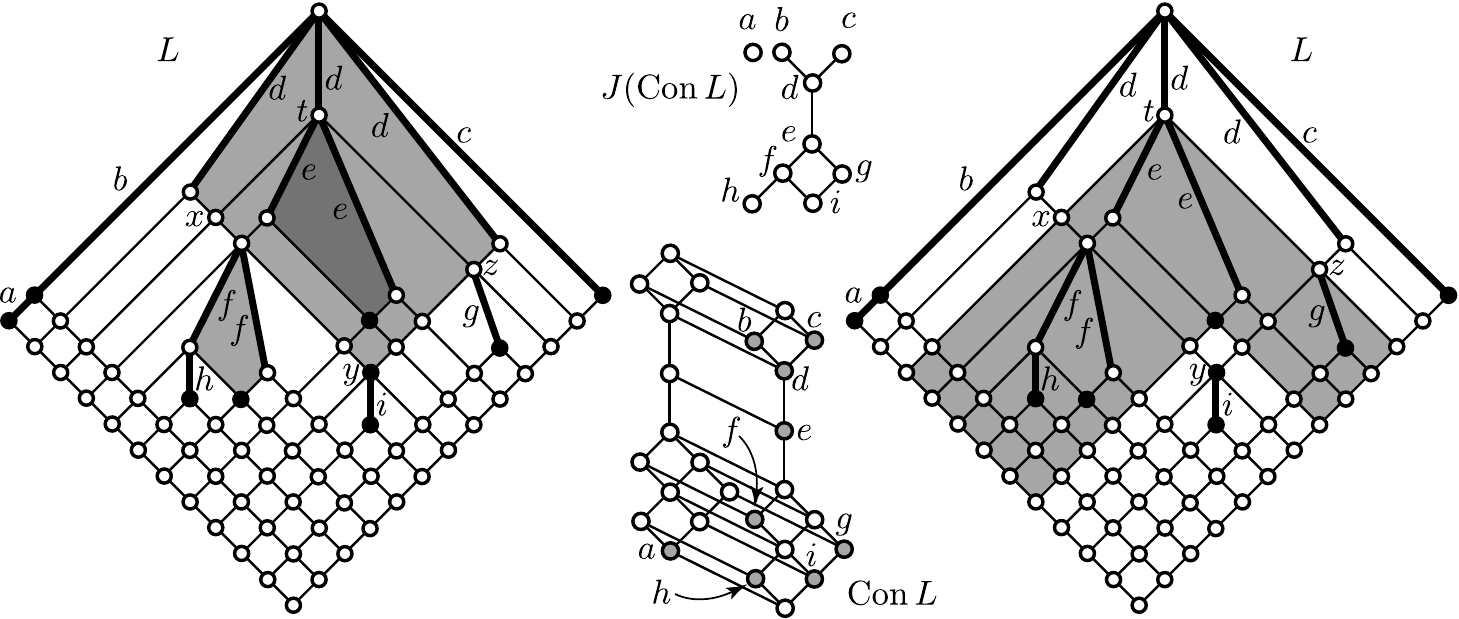}}
\caption{$\Lamps L\cong\Jir{\Con L}$, whence $\Lamps L$ determines $\Con L$}\label{figlot}
\end{figure}

\begin{definition}[line segments associated with lamps]\label{defpPrZwsklsmnszk}
Let $I:=[p,q]$ be a lamp of a slim rectangular lattice $L$ with a fixed $\bdia$-diagram, and let $F$ stand for the full geometric rectangle of $L$. Let $\tuple{p_x,p_y}\in\mathbb R^2$ and $\tuple{q_x,q_y}\in\mathbb R^2$ be the geometric points corresponding to $p=\foot I$ and $q=\Peak I$. As usual, $\Rplu$ will stand for the set of non-negative real numbers.  
We define the following four (geometric) line segments of normal slopes; see Figure~\ref{figslns} where these line segments are  dashed edges of normal slopes.
\allowdisplaybreaks{
\begin{align*}
\LRoof I&:=\set{\pair\xi\eta\in F: (\exists t\in\Rplu)\,( \xi =q_x-t\text{ and }\eta=q_y-t  )},\cr
\RRoof I&:=\set{\pair\xi\eta\in F: (\exists t\in\Rplu)\,( \xi =q_x+t\text{ and }\eta=q_y-t  )},\cr
\LFloor I&:=\set{\pair\xi\eta\in F: (\exists t\in\Rplu)\,( \xi =p_x-t\text{ and }\eta=p_y-t  )},\cr
\RFloor I&:=\set{\pair\xi\eta\in F: (\exists t\in\Rplu)\,( \xi =p_x+t\text{ and }\eta=p_y-t  )}.
\end{align*}}%
These line segments are called the \emph{left roof}, the \emph{right roof}, the \emph{left floor}, and the \emph{right floor} of $I$, respectively.  
Note that $\LRoof I$ and $\LFloor I$ lie on the same geometric line if and only if $I$ is a boundary lamp on the left boundary, and analogously for $\RRoof I$ and $\RFloor I$. We defined the \emph{roof of $I$} and the \emph{floor of $I$} as follows:
\begin{align*}
\Roof I &:=\LRoof I\cup \RRoof I,\text{ and}\cr
\Floor I&:=\LFloor I\cup \RFloor I;
\end{align*}
they are \ashape-shaped broken lines (that is, chevrons pointing upwards).
\end{definition}

\begin{figure}[htb]
\centerline
{\includegraphics[scale=0.9]{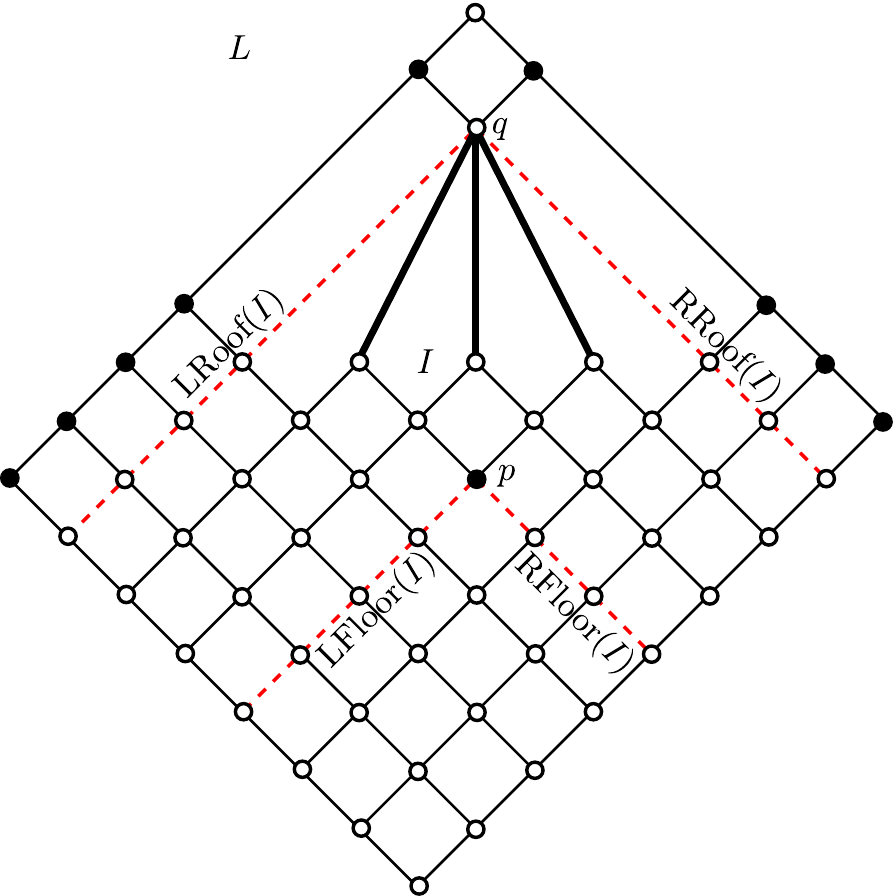}}
\caption{Four line segments associated with $I$}\label{figslns}
\end{figure}

In real life, neon tubes and lamps are for \emph{illuminating} in the sense of emitting light beams. Our lamps do this only downwards with normal slopes; the photons they emit can only go to the directions $\pair 1{-1}$ and $\pair{-1}{-1}$. Definition~\ref{defenlstS} below describes this more precisely. Note at this point that in addition to brightening with light, ``illuminating'' also means intellectual enlightening, that is, making things clear for human mind. This explains that lamps and neon tubes occur in our terminology just introduced: by \emph{illuminating} a part of a $\bdia$-diagram in a visual geometrical way like in physics, lamps 
also \emph{illuminate} the congruence structure by enlightening it in intellectual sense.  Convention~\ref{convxpllzT} is still in effect. 

\begin{definition}[Illuminated sets]\label{defenlstS}
Let $I:=[\foot I,q]=[\foot I,\Peak I]$ be a lamp of $L$.
A geometric point $\pair x y$ of the full geometric rectangle of $L$ is \emph{illuminated by $I$ from the left} 
if the lamp has a neon tube $[p_i, q]$ such that the edge $[p_i,q]$ as a geometric line segment 
has a nonempty intersection with the half-line $\set{\tuple{x-t, y+t}: 0\leq t\in \mathbb R}$. 
Similarly, a point $\pair x y$ of the full geometric rectangle of $L$ is \emph{illuminated by $I$ from the right} if
the half-line $\set{\tuple{x+t, y+t}: 0\leq t\in \mathbb R}$ has a nonempty intersection with at least one of the neon tubes of $I$.
If $\pair x y$ is illuminated from the left or from the right, then we simply say that this point is \emph{illuminated by the lamp $I$}. The set of points illuminated by the lamp $I$, that of points illuminated by $I$ from the right, and that from the left are denoted by 
\begin{equation}\left.
\begin{aligned}
\Enl I&=\Enl{[\foot I ,\Peak I]},\cr 
\LEnl I&=\LEnl{[\foot I ,\Peak I]},\text{ and} \cr
\REnl I&=\REnl{[\foot I ,\Peak I]}
\end{aligned}\,\,\right\}
\label{eqLngjnLnGjsT}
\end{equation}
respectively. The acronym ``Lit'' and its variants come from ``light'' rather than from ``illuminate''.  (The outlook of ``Ill'' coming from ``illuminate'' would not be satisfactory and would heavily depend on the font used.)
Let us emphasize that, say, $\LEnl I$ consist of points illuminated from the \emph{right}; the notation is explained by the fact that the geometric points of $\LEnl I$ are on the left of (and down from) $I$. 
Note that $\LEnl I$ is of positive geometric area if and only if $I$ is not a boundary lamp on the left boundary, and analogously for $\REnl I$. 
Finally, we also define
\begin{equation}
\pEnl I := \Enl I\setminus\Floor I.
\end{equation}
Note that $\Enl I$ is $\LEnl I\cup \REnl I$. By Definition~\ref{defpPrZwsklsmnszk} and \ref{defenlstS},
\begin{equation}
\parbox{8.7cm}{$\Enl I$ is (topologically) bordered by $\Roof I$, $\Floor I$, and appropriate line segments of $\leftb L$ and $\rightb L$, and so it is bordered by line segments of normal slopes.}
\label{eqpbxnfhRmsrlmrtlTslVz}
\end{equation}  
Note also that the intersection $\LEnl I\cap \REnl I$ can be of positive (geometric) area in the plane and that both $\LEnl I$ and $\REnl I$ are of positive area if and only if $I$ is an internal lamp.
\end{definition}

\begin{figure}[htb]
\centerline
{\includegraphics[width=\textwidth]{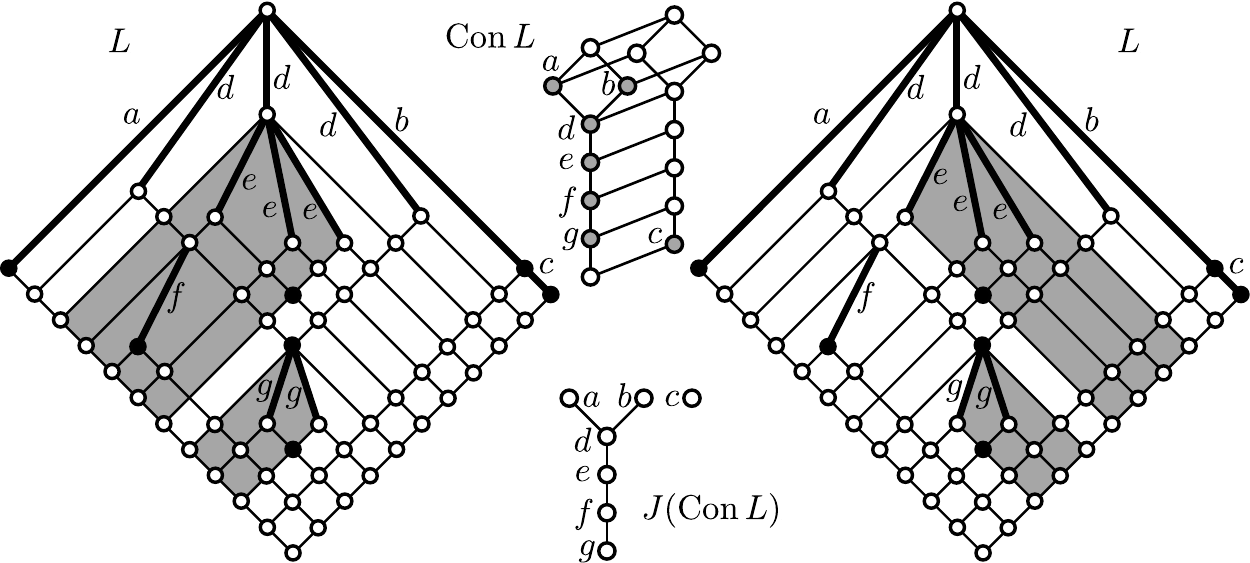}}
\caption{Illustrating \eqref{eqLngjnLnGjsT} and $\Lamps L\cong\Jir{\Con L}$}\label{figleri}
\end{figure}

For example, each of $\Sn 1$, $\Sn 2$, $\Sn 3$, and $\Sn 4$ of Figure~\ref{figsn} has a unique internal lamp, namely, the interval spanned by the  black-filled enlarged element in the middle and $1$. The illuminated set of this lamp is the ``\ashape-shaped'' grey-filled hexagon (containing light-grey and  dark-grey points).
Also, $\Sn n$ has exactly two boundary lamps and the illuminated set of each of these two lamps is the whole geometric rectangle of $\Sn n$, for every $n\in\Nplu$. If $E$ denotes the lamp with two $e$-labelled neon tubes on the right of Figure~\ref{figlot}, then $\Enl I$ is the ``\ashape-shaped'' grey-filled hexagon. In Figure~\ref{figleri}, let $E$ and $G$ denote the lamps consisting of the $e$-labelled edges and  the $g$-labelled edges, respectively. On the left of this figure, $\LEnl E$ and $\LEnl G$ are the grey-filled trapezoids while the grey-filled trapezoids on the right are $\REnl E$ and $\REnl G$. 
Let us note that, for any slim rectangular lattice $L$,  
\begin{equation}
\text{two distinct internal lamps can never have the same peak,}
\label{eqtxtxzGrhcmnZnrGmm}
\end{equation}
but all the three lamps, two boundary and one internal, have the same peak in $\Sn n$. 

Next, we introduce some relations on the set of lamps. Even though not all of these relations are applied in the subsequent sections, they and Lemma~\ref{lemmamain} will hopefully be useful in future investigation of the congruence lattices of slim planar semimodular lattices; for example, in 
Cz\'edli and Gr\"atzer~\cite{czgginprepar}.

\begin{definition}[Relations defined for lamps]\label{defrhRhs} 
Let $L$ be a slim rectangular lattice with a fixed $\bdia$-diagram. The set of lamps of $L$ will be denoted by $\Lamps L$.
On this set, we define eight irreflexive binary relations; seven in geometric ways based on Definitions~\ref{defgshscdzrkrrl}--\ref{defenlstS} and one in a purely algebraic way; these relations will soon be shown to be the same. For $I,J\in \Lamps L$,
\begin{enumeratei}
\item\label{defrhRhsa} 
let $\pair I J\in\rhgeomb$ mean that $I\neq J$,  $\body I\subseteq \Enl J$, and $I$ is an internal lamp;
\item\label{defrhRhsb} 
let $\pair I J\in\rhgeomc$ mean that $I$ is an internal lamp, $\cirrec I\subseteq \Enl J$, and $I\neq J$; 
\item\label{defrhRhsc} 
let $\pair I J\in\rhalg$ mean that  $\Peak I \leq \Peak J$  but $\foot I\not\leq \foot J$;  
\item\label{defrhRhsd} 
let $\pair I J\in\lrhgeomb$ mean that $I$ is an internal lamp, $I\neq J$, and 
\[\text{$\body I\subseteq \LEnl J\,$ or $\,\body I\subseteq \REnl J$;}
\]
\item\label{defrhRhse} 
let $\pair I J\in\lrhgeomc$  mean that $I$ is an internal lamp,  $I\neq J$, and 
\[\text{$\cirrec I\subseteq \LEnl J\,$ or $\,\cirrec I\subseteq \REnl J$;}
\]
\item\label{defrhRhsf} 
let $\pair I J\in\rhfoot$  mean that  $I\neq J$, $\foot I\in\Enl J$, and $I$ is an internal lamp;
\item\label{defrhRhsg} 
let $\pair I J\in\rhinfoot$  mean that  $I\neq J$ and $\foot I$ is in the geometric (or, in other words, topological) interior of $\Enl J$; and, finally,
\item\label{defrhRhsh}  let $\pair I J\in\rhinpfoot$  mean that  $\foot I\in\pEnl J$.
\end{enumeratei}
\end{definition}

\begin{remark}\label{remarkZhsmWgshp} In each of \eqref{defrhRhsa},\ \eqref{defrhRhsb},\dots,\eqref{defrhRhsh} of Definition~\ref{defrhRhs}, $I$ is an internal lamp and $I\neq J$;  this follows easily from other stipulations even where this is not explicitly mentioned.
\end{remark}

Now we are in the position to formulate the main result of this section. The congruence generated by a pair $\pair x y$ of elements will be denoted by $\con(x,y)$. If $\inp=[x,y]$ is an interval, then we can write $\con(\inp)$ instead of $\con(x,y)$.

\begin{lemma}[Main Lemma]\label{lemmamain} 
If $L$ is a slim rectangular lattice with a fixed $\bdia$-diagram, then the following four assertions hold.
\begin{enumeratei}
\item\label{lemmamaina}  The relations $\rhgeomb$, $\rhgeomc$,  $\rhalg$, $\lrhgeomb$, $\lrhgeomc$, $\rhfoot$,  $\rhinfoot$, and  $\rhinpfoot$  are all the same.
\item\label{lemmamainb}  The reflexive transitive closure of $\rhalg$  is a partial ordering of the set $\Lamps L$ of all lamps of $L$; we denote this reflexive transitive closure by $\leq$. 
\item\label{lemmamainc}  The poset $\tuple{\Lamps L;\leq}$ is isomorphic to the poset $\tuple{\Jir {\Con L};\leq}$ of nonzero join-irreducible congruences of $L$ with respect to the ordering inherited from $\Con L$, and the map
$\phi\colon \Lamps L\to \Jir{\Con L}$, defined by $[p,q]\mapsto \con{(p, q)}$, is an order isomorphism.
\item\label{lemmamaind} If $I\prec J$ (that is, $J$ covers $I$) in $\Lamps L$, then $\pair I J\in\rhalg$. 
\end{enumeratei}
\end{lemma}

The proof of this lemma heavily relies upon Cz\'edli~\cite{czgtrajcolor} and \cite{czgrectectdiag}. Before presenting this proof,  we need some preparations. First, we need to recall the \emph{multifork structure theory} from Cz\'edli~\cite{czgtrajcolor}. Minimal regions of a planar lattice are called \emph{cells}. Every  slim planar semimodular  lattice, in particular, every slim rectangular lattice $L$ is a \emph{$4$-cell lattice}, that is, its cells are formed by four edges; see Gr\"atzer and Knapp~\cite{gratzerknapp1}. So the cells of
$L$ are of the from $[a\wedge b, a\vee b]$ such that $a\parallel b$ and  $[a\wedge b, a\vee b]=\set{a\wedge b, a, b,  a\vee b}$. This cell is said to be a \emph{distributive cell} if the principal ideal $\ideal(a\vee b)$ is a distributive lattice. Let $n\in \Nplu$. To obtain the \emph{multifork extension $($of rank $n)$} of $L$ \emph{at a distributive $4$-cell} $J$ means that we change $J$ to a copy of $\Sn n$ and keep adding new elements while going to the southeast and the southwest to preserve semimodularity. 
This is visualized by Figure~\ref{figmfXtns}, where $L$ is drawn on the left, $J$ is the grey-filled 4-cell, $n=3$, and the slim rectangular lattice $L'$ we obtain by the multifork extension of $L$ at $J$ of rank 3 is $L'$, drawn on the right. The new elements, that is, the elements of $L'\setminus L$, are the pentagon-shaped ones. 
(For a more detailed definition of a multifork extension, which we do not need here since Figures~\ref{figmfXtns} and \ref{figmlsrNgrt} are sufficient for our purposes, the reader can resort to Cz\'edli~\cite{czgtrajcolor}. Note that \cite{czgtrajcolor} uses the term ``$n$-fold''  rather than ``of rank $n$''.)

\begin{figure}[htb]
\centerline
{\includegraphics[width=\textwidth]{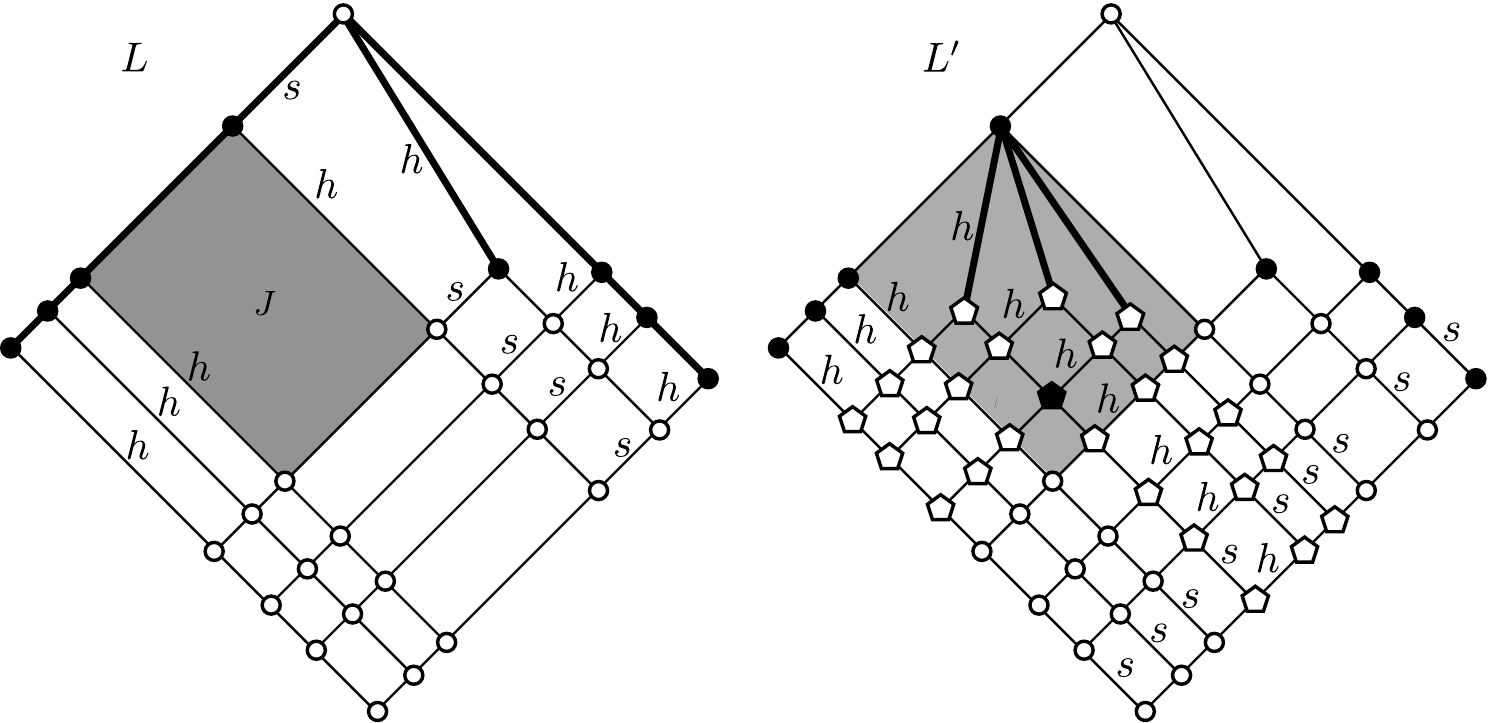}}
\caption{$L'$ is the multifork extension of rank 3 at the 4-cell $J$}\label{figmfXtns}
\end{figure}

By a \emph{grid} we mean the direct product of two finite nonsingleton chains or a $\bdia$-diagram of such a direct product. Note that a slim rectangular lattice is distributive if and only if it is a grid. 
Note also that, in a slim rectangular lattice with a fixed $\bdia$-diagram, 
\begin{equation}
\parbox{7.3cm}{a $4$-cell $I=[p,q]$ is distributive if and only if every edge in the ideal $\ideal q$ is of a normal slope;}
\label{eqpbxJrNfhgtWcBvmb}
\end{equation}
the \idez{only if} part of \eqref{eqpbxJrNfhgtWcBvmb} is Corollary 6.5 of Cz\'edli~\cite{czgrectectdiag} while the \idez{if} part follows easily by using that if all  edges of $\ideal q$ are of normal slopes then $\ideal q$ is a (sublattice of a) grid.

\begin{lemma}[Theorem 3.7 of Cz\'edli~\cite{czgtrajcolor}]\label{lemmamrkXtnDsZs}
Each slim rectangular lattice is obtained from a grid
by a sequence of multi-fork extensions at distributive $4$-cells, and every lattice obtained in this way is a slim rectangular lattice.
\end{lemma}

Figure~\ref{figmlsrNgrt} illustrates how we can obtain the lattice $L$ defined by Figure~\ref{figlot} from the initial grid $L_0$ in six steps. For $i=1,2,\dots 6$, $L_i$ is obtained from $L_{i-1}$ by  a multifork extension at the grey-filled 4-cell of $L_{i-1}$. 
We still need one important concept, which we recall from Cz\'edli~\cite{czgtrajcolor}; it was originally introduced in  Cz\'edli and Schmidt~\cite{czgschtJH}.

\begin{figure}[htb]
\centerline
{\includegraphics[width=\textwidth]{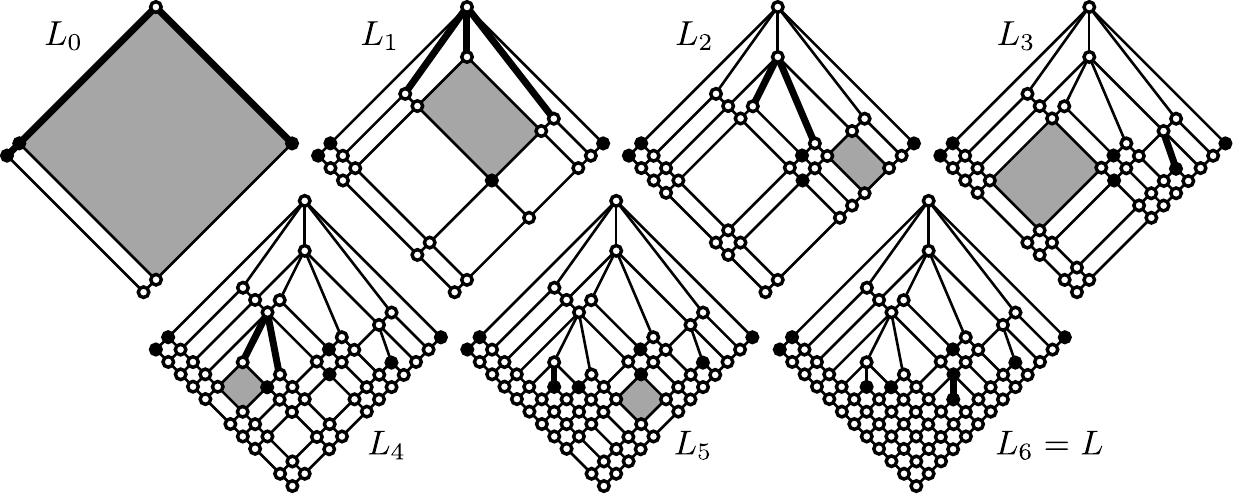}} 
\caption{Illustrating Lemma~\ref{lemmamrkXtnDsZs}: a sequence of multifork extensions}\label{figmlsrNgrt}
\end{figure}

\begin{definition}[Trajectories]\label{deftrajectory}
Let $L$ be a slim rectangular lattice with a fixed $\bdia$-diagram. The set of its \emph{edges}, that is, the set of its prime intervals is denoted by $\Edges L$. We say that $\inp,\inq\in\Edges L$  are \emph{consecutive edges} if they are opposite sides of a 4-cell. Maximal sequences of consecutive edges are called \emph{trajectories}. 
In other words, the blocks of the least equivalence relation on $\Edges L$ including the  consecutiveness relation are called trajectories. 
If a trajectory has an edge on the upper boundary (equivalently, if it has an edge that is the unique neon tube of a boundary lamp), then this trajectory a \emph{straight trajectory}. Otherwise, it is a \emph{hat trajectory}. Each trajectory has a unique edge that is a neon tube; this edge is called the \emph{top edge} of the trajectory. The top edge of a trajectory $u$ will be denoted by $\topedge u$ while $\Trajs L$ will stand for the set of trajectories of $L$.
\end{definition}

For example, if $L$ is the lattice on the left of Figure~\ref{figmfXtns}, then it has eight trajectories. One of the eight trajectories is a hat trajectory and consists of the $h$-labelled edges. There are seven straight trajectories and one of these seven consists of the $s$-labelled edges. 
Note that all neon tubes of $L$ are drawn by thick lines. On the right of the same figure, only the neon tubes of the new lamp are thick. Also, $L'$ has exactly four hat trajectories and seven straight trajectories. One of the hat trajectories consists of the $h$-labelled edges while the $s$-labelled edges form a straight trajectory.

\begin{proof}[Proof of Lemma \ref{lemmamain}]
Let $L$ be a slim rectangular lattice with a fixed $\bdia$-diagram.
To prove part \eqref{lemmamaina}, we need some preparations. We know from Lemma~\ref{lemmamrkXtnDsZs} that $L$ is obtained  by a 
\begin{equation}
\parbox{7.5cm}{sequence $L_0, L_1,\dots, L_k=L$ such that $L_0$ is a grid and $L_i$ is a  multifork extension of $L_{i-1}$ at a distributive 4-cell $H_i$ of $L_{i-1}$ for $i\in\set{1,\dots,k}$.}
\label{eqpbxHzwsrZstnC}
\end{equation}
This is illustrated by Figure~\ref{figmlsrNgrt}.
In $\Lamps{L_0}$, there are only boundary lamps. 
It is clear by definitions that each internal lamp arises from the replacement of the distributive 4-cell $H_i$ of $L_{i-1}$, grey-filled in Figures~\ref{figmfXtns} and \ref{figmlsrNgrt}, by a copy of $\Sn{n_i}$, for some $i\in\set{1,\dots, k}$ and $n_i\in\Nplu$. Shortly saying,
\begin{equation}
\parbox{7.7cm}{every internal lamp comes to existence from a multifork extension. Furthermore, if a lamp $K$ comes by a multifork extension at a 4-cell $H_i$, then $\cirrec I$ is the geometric region determined by $H_i$;}
\label{eqtxtmpFrRHhW}
\end{equation}
the second half of \eqref{eqtxtmpFrRHhW} follows from \eqref{eqpbxJrNfhgtWcBvmb} and Convention~\ref{convxpllzT}.
For example, on the right of Figure~\ref{figmfXtns}, the new lamp what the multifork extension has just brought  is the 
lamp with pentagon-shaped black-filled foot 
 and the thick edges are its neon tubes. Similarly, the
new lamp is the one with thick neon tube(s) in each of $L_1$, \dots, $L_6=L$ in Figure~\ref{figmlsrNgrt}. Keeping Convention~\ref{convxpllzT} in mind, it is clear that if $[p,q]$ is the new lamp that the multifork extension of $L_{i-1}$ brings into $L_i$, then 
$[\cornl{L_i}]\wedge p, p]$ and $[\cornr{L_i}]\wedge p, p]$ are chains with all of their edges being of normal slopes. Observe that
\begin{equation}
\parbox{8.5cm}
{no geometric line segment that consists of some edges of $L_i$ disappears at further multifork extension steps,}
\label{eqpbxZtSrptnK}
\end{equation}
but there can appear more vertices on it. 
This fact, Convention~\ref{convxpllzT}, \eqref{eqpbxJrNfhgtWcBvmb},   \eqref{eqtxtmpFrRHhW}, 
and the fact that $L_{i-1}$ is a sublattice of its multifork extension $L_i$ for all $i\in\set{1,\dots,k}$ yield that 
\begin{equation}
\parbox{9,4cm}{if $I=[p,q]\in \Lamps L$, 
then the lowest point of $\LRoof I$ and that of $\LFloor I$ are 
$\cornl{L}\wedge q$ and $\cornl{L}\wedge p$, respectively, and the intervals $[\cornl{L}\wedge q,q]$ and $[\cornl{L}\wedge p,p]$ are chains. That is, $\LRoof I$ and $\LFloor I$ correspond to intervals that are chains. The edges of these chains are of the same normal slope. Similar statements
hold for \idez{right} instead of \idez{left}.}
\label{eqpbxZhRsjPzrSt}
\end{equation}
For later reference, note another easy consequence of \eqref{eqpbxHzwsrZstnC} and \eqref{eqpbxZtSrptnK}, or \eqref{eqpbxZhRsjPzrSt}:
\begin{equation}
\parbox{9.7cm}{With reference to \eqref{eqpbxHzwsrZstnC}, assume that $i<j\leq k$ and a lamp $I$ is present in $L_i$. Then $\Enl I$ is the same in $L_i$ as in $L_j$ and $L$.}
\label{eqpbxdhzWnhRgbBBsWt}
\end{equation}
For an internal lamp $I$, the leftmost neon tube and the (not necessarily distinct) rightmost neon tube of $I$ are the right upper edge and the left upper edge of a 4-cell, respectively.  
Combining this fact,  \eqref{eqtxtmpFrRHhW}, and Observation 6.8(iii) of Cz\'edli~\cite{czgrectectdiag}, we conclude that 
\begin{equation}
\parbox{8.4cm}{the two upper (geometric) sides of the circumscribed rectangle $\cirrec I$ of an internal lamp $I$ are edges (that is, prime intervals) of $L$.}
\label{eqpbxwizgRpsrbrWnGskS}
\end{equation}
Since $\LEnl J\subseteq \Enl J$ and $\REnl J\subseteq \Enl J$ for every $J\in \Lamps L$, 
\begin{equation}
\lrhgeomc\subseteq \rhgeomc \quad\text{and}\quad \lrhgeomb\subseteq \rhgeomb.  
\label{eqsspzRtmQcpxVcB}
\end{equation}

Next, we are going to prove that 
\begin{equation}
\rhgeomc\subseteq \rhgeomb\quad\text{and}\quad \lrhgeomc\subseteq \lrhgeomb.
\label{eqsizrskljDmNdGmrc}
\end{equation}
To do so, assume that $\pair I J\in\rhgeomc$ and $\pair {I'}{J'} \in\lrhgeomc$. This means that $\cirrec I\subseteq \Enl J$ and 
$\cirrec {I'}\subseteq \LEnl {J'}$ or $\cirrec {I'}\subseteq \REnl {J'}$, respectively.  
The definition of $\Sn n$ and multifork extensions,  \eqref{eqtxtmpFrRHhW}, and \eqref{eqpbxZtSrptnK} yield that $\body I\subseteq \cirrec I$ and $\body{I'}\subseteq \cirrec {I'}$. Combining these inclusions with the earlier ones, 
we have that  $\body I\subseteq \Enl J$ and  $\body {I'}\subseteq \LEnl {J'}$ or  $\body {I'}\subseteq \REnl {J'}$. Hence,  
$\pair I J\in\rhgeomb$ and $\pair {I'} {J'}\in\lrhgeomb$, proving the validity of \eqref{eqsizrskljDmNdGmrc}.

We claim that 
\begin{equation}
\rhgeomb\subseteq \rhalg .
\label{eqprmsdrWsvThnNn}
\end{equation}

To show this, assume that $\pair I J\in\rhgeomb$, that is, $\body I\subseteq \Enl J$, $I\neq J$,  and $I$ is an internal lamp.  We know from 
\eqref{eqpbxnfhRmsrlmrtlTslVz} that $\Enl J$ is bordered 
by geometric line segments of normal slopes.
Hence, Corollary 6.1 of Cz\'edli~\cite{czgrectectdiag} and 
$\body I\subseteq \Enl J$ yield that $\Peak I\leq \Peak J$. Figure~\ref{figsn} and Convention~\ref{convxpllzT} imply that $\foot I$ is not on $\Floor J$,
the lower geometric boundary of $\Enl J$. It is trivial 
by $\body I\subseteq \Enl J$ that $\foot I$ is not (geometrically and strictly) below $\Floor J$.
Hence, the just mentioned Corollary 6.1 of \cite{czgrectectdiag} shows that $\foot I\not\leq \foot J$. We have obtained that 
$\pair I J\in\rhalg$. Thus, $\rhgeomb\subseteq \rhalg$, proving \eqref{eqprmsdrWsvThnNn}.

\begin{figure}[htb]
\centerline
{\includegraphics[scale=0.9]{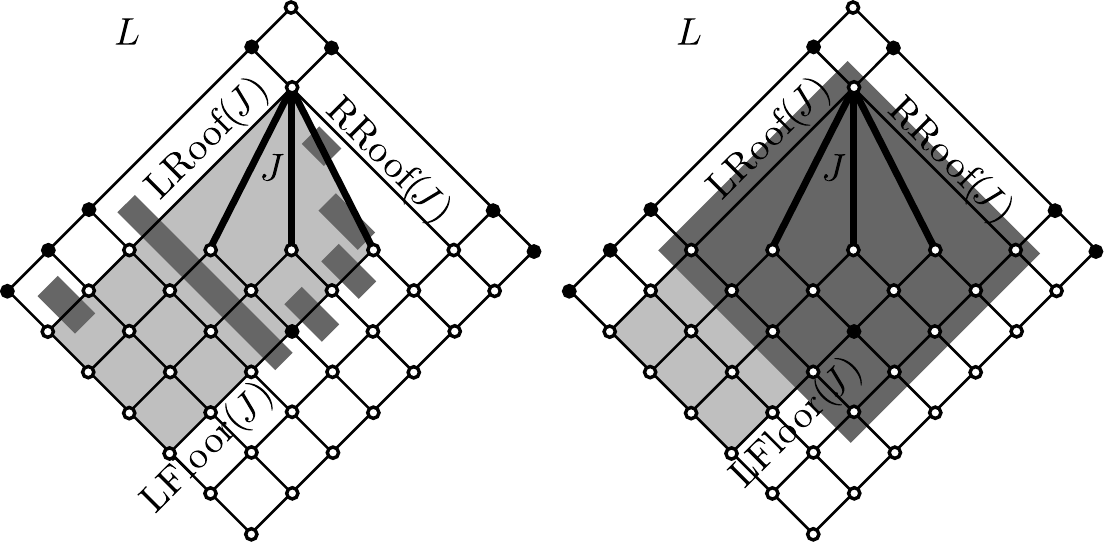}}
\caption{Proving that $ \rhfoot \subseteq \lrhgeomc$}\label{figftnt}
\end{figure}

By construction, see Figure~\ref{figsn}, $\cirrec I$ contains $\foot I$ as a geometric point in its topological interior for every internal lamp $I\in\Lamps L$. This clearly yields the first  inclusion in \eqref{eqdtmLzNvrSbRphhG} below:
\begin{equation}
\rhgeomc \subseteq \rhinfoot \subseteq \rhinpfoot \subseteq \rhfoot \subseteq \lrhgeomc. 
\label{eqdtmLzNvrSbRphhG}
\end{equation}
The second and the third inclusions above are trivial. 
In order to show the fourth inclusion, $\rhfoot\subseteq \lrhgeomc$, assume that
$\pair I J\in\rhfoot$. 
We know that $I\neq J$, $I$ is an internal lamp, and $\foot I\in\Enl J$. Since $\Enl J=\LEnl J\cup\REnl J$, 
left-right symmetry allows us to assume that the geometric point 
$\foot I$ belongs to $\LEnl J$. But, as it is clear from Figure~\ref{figsn} and \eqref{eqtxtmpFrRHhW}, $\foot I$ is in the (topological) interior of $\cirrec I$. Hence, 
\begin{equation}
\text{there is an open set $U\subseteq \mathbb R^2$ such that 
$U\subseteq \cirrec I\cap \LEnl J$.}
\label{eqtxtHvstrnplttmprKrr}
\end{equation}
We claim that $\cirrec I\subseteq \LEnl J$. Suppose, for a contradiction, that this inclusion fails. 
We know from \eqref{eqpbxnfhRmsrlmrtlTslVz}, \eqref{eqtxtmpFrRHhW},  \eqref{eqpbxZtSrptnK}, \eqref{eqpbxdhzWnhRgbBBsWt},  and Remark~\ref{remarksZsSTmVbQxPT} that $\LEnl J$ is surrounded by edges of normal slopes except its top right precipitous side. Also, 
we obtain from \eqref{eqpbxJrNfhgtWcBvmb}, \eqref{eqtxtmpFrRHhW}, and  \eqref{eqpbxwizgRpsrbrWnGskS} that $\cirrec I$ is a rectangle whose sides are of normal slopes, and the upper two sides are edges (that is, prime intervals).  These facts, \eqref{eqtxtHvstrnplttmprKrr}, and $\cirrec I\not\subseteq \LEnl J$ yield that a side of $\cirrec I$ crosses a side of $\LEnl J$. But edges do not cross in a planar diagram, whence no side of $\LEnl J$ crosses an upper edge of $\cirrec I$. So if we visualize $\LEnl J$ by light-grey color and $\cirrec I$ is dark-grey, then none of the six little dark-grey rectangles on the left of Figure~\ref{figftnt} can be $\cirrec I$. Note that, in both parts of this Figure, $J$ is formed by the three thick neon tubes. 
Since the above-mentioned six dark-grey rectangles represent generality, the only possibility is that the top vertex $\Peak I$ of $\cirrec I$ is positioned like the top of the dark-grey rectangle on the right of  Figure~\ref{figftnt}.  (There can be lattice elements not indicated in the diagram, and we took into consideration  that the top sides of $\cirrec J$ are also edges and cannot be crossed and that \eqref{eqtxtxzGrhcmnZnrGmm} holds). But then Corollary 6.1 of Cz\'edli~\cite{czgrectectdiag} yields that $\Peak J\leq \Peak I$.  With reference to \eqref{eqpbxHzwsrZstnC} and  \eqref{eqtxtmpFrRHhW}, let $i$ and $j$ denote the subscripts such that $I$ and $J$ came to existence in $L_i$ and $L_j$, respectively. Since 
none of the lattices $\Sn n$, $n\in\Nplu$,  is distributive, $j>i$. 
So the elements of $L_i$ are old when $L_j$ is constructed as a multifork extension of $L_{j-1}$. Here by an old element of $L_j$ we mean an element of $L_{j-1}$. 
It is clear by the concept of multifork extensions and that of $\LEnl J$,  \eqref{eqpbxJrNfhgtWcBvmb}, and  \eqref{eqtxtmpFrRHhW}  that no old element belongs to the topological interior of $\Enl J$. But 
$\foot I\in L_i\subseteq L_{j-1}$ does belong to this interior, which is a contradiction showing that $\cirrec I\subseteq \LEnl J$. Thus, 
$\pair I J\in \lrhgeomc$, and we obtain that $\rhfoot\subseteq \lrhgeomc$. This completes the proof of \eqref{eqdtmLzNvrSbRphhG}.

Next, we are going show that
\begin{equation}
\rhalg\subseteq \rhfoot.
\label{eqszGrbdlkkFhvN}
\end{equation}
To verify this, assume that $\pair I J\in\rhalg$. We know that $\Peak I\leq \Peak J$ but $\foot I\not\leq \foot J$. Observe that $\foot I\leq \Peak I\leq \Peak J$. 
We know that $\Peak J$ and $\foot J$ are the vertices of the \ashape-shaped $\Roof J$ and $\Floor J$, respectively, and $\Roof J$ and $\Floor J$ consist of line segments of normal slopes. Hence, 
it follows from Corollary 6.1 of Cz\'edli~\cite{czgrectectdiag} and $\foot I\leq \Peak J$ that $\foot I$ is geometrically below or on $\Roof J$. On the other hand, Corollary 6.1 of  \cite{czgrectectdiag} 
and $\foot I\not\leq \foot J$ imply that $\foot I$ is neither geometrically below, nor on $\Floor J$. So $\foot I$ is above $\Floor J$. 
Therefore, $\foot I$ is geometrically between $\Floor J$ and $\Roof J$. Thus, 
 \eqref{eqpbxnfhRmsrlmrtlTslVz} gives that  $\foot I\in\Enl J$. This shows that $\pair I J\in\rhfoot$, and we have shown the validity of \eqref{eqszGrbdlkkFhvN}.

\begin{figure}[htb]
\centerline
{\includegraphics[scale=1.0]{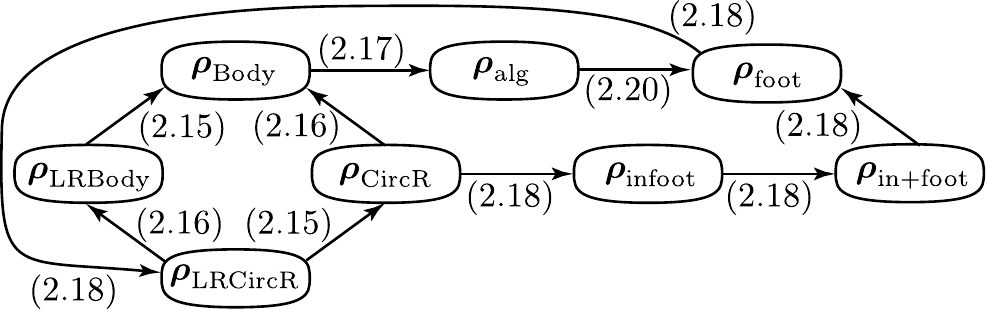}}
\caption{An overview of what has already been proved \label{figschm}}
\end{figure}

The directed graph in Figure~\ref{figschm}  visualizes what we have already shown.  Each (directed) edge  $\rho_1 \to \rho_2$ of the graph
means that  $\rho_1\subseteq \rho_2$ has been formulated in the displayed equation given by the label of the edge.
The directed graph is strongly connected,  implying part \eqref{lemmamaina}.

Next, we turn our attention to parts \eqref{lemmamainb} and \eqref{lemmamainc}.
For a trajectory $u$ of $L$, $\topedge u$ is a neon tube, so it belongs to a unique lamp; we denote this lamp by $\Lmp u$. Lamps are special intervals. Hence, in agreement with how the notation $\con()$ was introduced
 right before Lemma~\ref{lemmamain}, $\con(\Lmp u)$ will denote the congruence  $\con(\foot{\Lmp u} , \Peak{\Lmp u})$ generated by the interval $\Lmp u$. We claim that, for any trajectory $u$ of $L$, 
\begin{equation}
\con(\topedge u)=\con(\Lmp u).
\label{eqdzhTslRjGnw}
\end{equation}
To show \eqref{eqdzhTslRjGnw}, observe that this assertion  is trivial if $\topedge u$ is the only neon tube of $\Lmp u$ since then the same prime interval generates the congruence on both sides of  \eqref{eqdzhTslRjGnw}.

Hence, we can assume that $\Lmp u$ has more than one neon tubes; clearly, then $\Lmp u$ is an internal lamp. Let $p:=\foot{\Lmp u}$, $q:=\Peak{\Lmp u}$, and let $[p_1,q]=\topedge u$, $[p_2,q]$, \dots, $[p_m,q]$ be a list of all neon tubes of $\Lmp u$. With this notation, \eqref{eqdzhTslRjGnw} asserts that $\con(p_1,q)=\con(p,q)$. 
Since the congruence blocks of a finite lattice are intervals  and $p\leq p_1\leq q$, 
the inequality $\con(p_1,q)\leq \con(p,q)$ is clear.  It follows in a straightforward way by inspecting the lattice $\Sn m$, see Figure~\ref{figsn}, or it follows trivially by the Swing Lemma, see 
 Gr\"atzer~\cite{ggswinglemma} or    Cz\'edli, Gr\"atzer and  Lakser~\cite{czggghlswing}, that $\con(p_1,q)=\con(p_2,q)=\dots=\con(p_m,q)$. Hence, $(p_i,q)\in \con (p_1,q)$ for $i\in\set{1,\dots,m}$, and we obtain that
\begin{equation}
(p,q)\overset{\eqref{eqdRsrFsPd}}{=}(p_1\wedge\dots\wedge p_m,q\wedge\dots\wedge q)\in\con(p_1,q).
\end{equation}
This yields the converse inequality $\con(p_1,q)\geq \con(p,q)$ and proves \eqref{eqdzhTslRjGnw}.

As usual,  the smallest element and largest element of an interval $I$ will often be denoted by $0_I$ and $1_I$, respectively.  Following Cz\'edli~\cite[Definition 4.3]{czgtrajcolor},
we define two relations  on $\Trajs L$. 
For trajectories $u,v\in \Trajs L$, we let 
\begin{equation}
\pair u v\in \bsigma\defiff \left\{\,\,
\parbox{6.5cm}{$1_{\topedge u}\leq 1_{\topedge v}$,   $0_{\topedge u}\not\leq 0_{\topedge v}$, and $u$ is a hat trajectory.}\right.
\label{eqzksrCsdrWksv}
\end{equation}
The  second relation is $\btau$, the reflexive transitive closure of $\bsigma$. We also need a third relation on $\Trajs L$; it is $\bTheta:=\btau\cap\btau^{-1}$. The $\bTheta$-block of a trajectory $u\in\Trajs L$ will be denoted by $u/\bTheta$. Since $\btau$ is a \emph{quasiordering}, that is, a reflexive and transitive relation, it is well known that the quotient set $\Trajs L/\bTheta$ turns into a poset $\tuple{\Trajs L/\bTheta, \btau/\bTheta}$  by defining
\begin{equation}
\tuple{u/\bTheta , v/\bTheta} \in \btau/\bTheta \defiff \pair u v\in \btau
\label{eqdsThBbTslw}
\end{equation}
for $u,v\in\Trajs L$; see, for example, (4.1) in Cz\'edli~\cite{czgtrajcolor}. We claim that, for any $u,v\in\Trajs L$, 
\begin{equation}
\tuple{u/\bTheta , v/\bTheta} \in \btau/\bTheta  \iff 
\Lmp u \leq  \Lmp v.
\label{eqszlmtjGmpNm}
\end{equation}
As the first step towards \eqref{eqszlmtjGmpNm}, we show that,
for any $u,v\in\Lamps L$,  
\begin{equation}
u/\bTheta = v/\bTheta \iff \Lmp u = \Lmp v.
\label{eqsRszTmrvPrl}
\end{equation}
The  $\Leftarrow$ part of \eqref{eqsRszTmrvPrl} is quite easy. Assume that $\Lmp u=\Lmp v$. Then it is clear by Figure~\ref{figsn} and \eqref{eqzksrCsdrWksv} that $\pair u v$ and $\pair v u$ belong to $\bsigma$. Hence, both $\pair u v$ and $\pair v u$ are in $\btau$, whence $u/\bTheta = v/\bTheta$. 

To show the converse implication,   assume that $u/\bTheta = v/\bTheta$ and, to exclude a trivial case,  $u\neq v$. This assumption gives that $\pair u v\in\btau$, and thus there is a $k\in\Nplu$ and there are elements $w_0=u$, $w_1$, \dots, $w_k=v$ such that $\tuple{w_{i-1},w_i}\in\bsigma$ for all $i\in\set{1,\dots,k}$. By \eqref{eqzksrCsdrWksv}, $u=u_0$ is a hat trajectory and 
$1_{\topedge u}=1_{\topedge {w_0}}\leq 1_{\topedge {w_1}}\leq \dots\leq 1_{\topedge {w_k}}=1_{\topedge {v}}$. That is, $1_{\topedge {u}}\leq 1_{\topedge {v}}$. Since $\pair v u$ is also in $\btau$, we also have that $v$ is a hat trajectory and $1_{\topedge {v}}\leq 1_{\topedge {u}}$. Hence, $1_{\topedge {v}}= 1_{\topedge {u}}$ and so $\Peak{\Lmp u}=\Peak{\Lmp v}$. Thus, $\Lmp u=\Lmp v$ by \eqref{eqtxtxzGrhcmnZnrGmm}, proving \eqref{eqsRszTmrvPrl}.

In the next step towards \eqref{eqszlmtjGmpNm},   \eqref{eqsRszTmrvPrl} allows us to assume that $u/\bTheta \neq v/\bTheta$, which is equivalent to $\Lmp u\neq \Lmp v$. First, we show that 
\begin{equation}
\text{if }\pair u v \in\bsigma,\text{ then } \Lmp u \leq \Lmp v.
\label{eqwRmNtnPsBw}
\end{equation}
Let $u=u_1,\dots, u_k$ and
$v=v_1,\dots, v_t$ be the neon tubes of $\Lmp u$ and $\Lmp v$, respectively. With reference to \eqref{eqpbxHzwsrZstnC} and  \eqref{eqtxtmpFrRHhW}, let $i$ and $j$ denote the subscripts such that $\Lmp u$ and $\Lmp v$ came to existence in $L_i$ and $L_j$, respectively. We obtain from $\pair u v \in\bsigma$ that $\Peak{\Lmp u}= 1_{\topedge {u}} \leq 1_{\topedge {v}}=\Peak{\Lmp v}$. By \eqref{eqzksrCsdrWksv}, $u$ is a hat trajectory, whence $\Lmp u$ is an internal lamp and so $i\geq 1$. Since none of the lattices $\Sn n$, $n\in\Nplu$,  is distributive, $i>j$. Since the sequence  \eqref{eqpbxHzwsrZstnC} is increasing, we obtain that $0_{\topedge v}\in L_{i-1}$. 
It is clear from 
\eqref{eqtxtmpFrRHhW} and the description of multifork extensions (see Figures~\ref{figmfXtns} and \ref{figmlsrNgrt}, and see also Figure~\ref{figsn}) that
\begin{equation}
\parbox{8.2cm}{if, according to \eqref{eqpbxHzwsrZstnC}, a lamp $K$ comes to existence in $L_\ell$, $x\in L_{\ell-1}$, and $\foot K\leq x$, then $\Peak K\leq x$.}
\label{eqpbxZsRxPwMskndRs}
\end{equation}
Suppose for a contradiction that $\foot{\Lmp u}\leq 0_{\topedge {v}}$. Applying \eqref{eqpbxZsRxPwMskndRs} with $\ell=i$ and $K=I$,  the already-mentioned $0_{\topedge v}\in L_{i-1}$
leads to $\Peak{\Lmp u}\leq 0_{\topedge {v}}$. But then
$0_{\topedge u}<1_{\topedge u}=\Peak{\Lmp u}\leq 0_{\topedge {v}}$, which is a contradiction since $0_{\topedge u}\not\leq0_{\topedge v}$ by  $\pair u v\in\bsigma$. Hence, $\foot{\Lmp u}\not\leq 0_{\topedge {v}}$. Combining this with 
$\foot{\Lmp v}\leq 0_{\topedge {v}}$, see \eqref{eqdRsrFsPd},
we obtain that $\foot{\Lmp u}\not\leq \foot{\Lmp v}$. We have already seen that $\Peak{\Lmp u} \leq \Peak{\Lmp v}$, whereby $\tuple{\Lmp u, \Lmp v}\in\rhalg$. This yields that 
$\Lmp u\leq \Lmp v$ since ``$\leq$'' is the reflexive transitive closure of $\rhalg$. Thus, we have shown the validity of \eqref{eqwRmNtnPsBw}.

Next, we assert that, for any $u,v\in\Trajs L$, 
\begin{equation}
\text{if }\tuple{\Lmp u, \Lmp v}\in \rhalg,\text{ then }
\tuple{ u,v}\in\btau.
\label{eqdzjlnTDseRkRkfkbRbfsn}
\end{equation}
To show this, let $v_1=v,v_2,\dots, v_k$ be the neon tubes of  $\Lmp v$. Assume that the pair $\tuple{\Lmp u, \Lmp v}$ belongs to $\rhalg$. Then    
\begin{equation}
\parbox{7.6cm}{$1_{\topedge {u}}=\Peak{\Lmp u}\leq \Peak{\Lmp v}=1_{\topedge {v_j}}$ for all 
 $j\in\set{1,\dots,k}$,}
\label{eqhszZsjtLpSJhHcbS}
\end{equation} 
and we know from Remark~\ref{remarkZhsmWgshp} that $\Lmp u$ is an internal lamp. Hence, $u$ is a hat trajectory.
On the other hand, $0_{\topedge u}\not\leq \foot{\Lmp v}$ since otherwise
$\foot{\Lmp u}\leq 0_{\topedge u}\leq \foot{\Lmp v}$ would contradict the containment
$\tuple{\Lmp u, \Lmp v}\in\rhalg$.  If we had that 
$0_{\topedge u}\leq 0_{\topedge {v_j}}$ for all $j\in\set{1,\dots,k}$, then we would obtain that
\[0_{\topedge u}\leq \bigwedge_{j\in\set{1,\dots,k}} 0_{\topedge {v_j}} \overset{\eqref{eqdRsrFsPd}}= \foot{\Lmp v},
\]
contradicting $0_{\topedge u}\not\leq \foot{\Lmp v}$. Hence, there exists a  $j\in\set{1,\dots,k}$ such that 
$0_{\topedge u}\not\leq 0_{\topedge {v_j}}$. Thus, using  
\eqref{eqhszZsjtLpSJhHcbS}, \eqref{eqzksrCsdrWksv}, and the fact that $u$ is a hat trajectory, we obtain that $\pair u{v_j}\in\bsigma$. Hence, $\pair u{v_j}\in\btau$. Also,  $\Lmp{v_j}=\Lmp v$ and \eqref{eqsRszTmrvPrl} yield that $\pair{v_j}v\in\btau$. By transitivity, $\pair u v\in\btau$, proving \eqref{eqdzjlnTDseRkRkfkbRbfsn}. 

By \eqref{eqdsThBbTslw}, what \eqref{eqszlmtjGmpNm} asserts is equivalent to the statement that $\pair u v\in\btau \iff \Lmp u\leq \Lmp v$. Since $\btau$ on $\Trajs L$ and ``$\leq$'' on $\Lamps L$ are the reflexive transitive closures of $\bsigma$ and $\rhalg$, respectively, the desired \eqref{eqszlmtjGmpNm} follows from \eqref{eqwRmNtnPsBw} and \eqref{eqdzjlnTDseRkRkfkbRbfsn}.

Next,  we are going to call a certain map a \emph{quasi-coloring}. (The exact meaning of a quasi-coloring was introduced in   Cz\'edli~\cite{czgreprhomr} and \cite{czgtrajcolor} but we do not need it here.)   
Following Definition 4.3(iv) of \cite{czgtrajcolor}, we define a map $\xi$ from the set $\Edges L$ of edges of $L$ to $\Trajs L$ by the rule $\inp \in\xi(\inp)$ for every $\inp \in \Edges L$. That is, $\xi$ maps an edge to the unique trajectory containing it. By Theorem 4.4 of \cite{czgtrajcolor}, $\xi$ is a quasi-coloring.   
It follows from Lemma 4.1 of \cite{czgtrajcolor} that 
\begin{equation}
\text{the posets }\tuple{\Jir L;\leq}\text{ and }\tuple{\Trajs L/\bTheta; \btau/\bTheta}\text{ are isomorphic}.
\label{eqzTmxVkBLkkRnjtSqr}
\end{equation}
By  \eqref{eqszlmtjGmpNm} and its particular case, \eqref{eqsRszTmrvPrl}, the structures
\begin{equation}
\tuple{\Trajs L/\bTheta, \btau/\bTheta}\text{ and }\tuple{\Lamps L; \leq}\text{ are also isomorphic}.
\label{eqzTmxprzsHrcvNts}
\end{equation}
We have already mentioned around \eqref{eqdsThBbTslw} and we also know from \eqref{eqzTmxVkBLkkRnjtSqr} that  $\tuple{\Trajs L/\bTheta, \btau/\bTheta}$ is a poset. 
Thus, $\tuple{\Lamps L; \leq}$ is also a poset by \eqref{eqzTmxprzsHrcvNts}, proving part \eqref{lemmamainb} of Lemma \ref{lemmamain}.

Combining \eqref{eqzTmxVkBLkkRnjtSqr} and \eqref{eqzTmxprzsHrcvNts}, it follows that $\tuple{\Jir L;\leq}\cong
\tuple{\Lamps L; \leq}$, which is the first half of part  \eqref{lemmamainc} of Lemma \ref{lemmamain}. 
It is straightforward to extract from  (2.6)--(2.8) of Cz\'edli \cite{czgreprhomr}, or from the proof of \eqref{eqzTmxVkBLkkRnjtSqr} or that of Theorem 7.3(i) of \cite{czgtrajcolor} that 
\begin{equation}
\parbox{8.9cm}{the map $\psi_1\colon \tuple{\Trajs L/\bTheta; \btau/\bTheta}\to \tuple{\Jir{\Con L};\leq}$ defined by $u/\bTheta\mapsto \con(\topedge u)$ is a poset isomorphism.}
\label{eqpbxZsfJklstrbBs}
\end{equation}
Observe that every lamp $I$ is of the form $\Lmp u$ for some $u\in\Trajs L$; indeed, we can choose $u$ as $\xi(\inp)$ for some (in fact, any) neon tube $\inp$ of $I$.
By \eqref{eqszlmtjGmpNm} and \eqref{eqsRszTmrvPrl}, 
\begin{equation}
\parbox{9cm}{the map $\psi_2\colon \tuple{\Lamps L;\leq} \to
\tuple{\Trajs L/\bTheta; \btau/\bTheta}$, defined by
$\Lmp u\mapsto u/\bTheta$, is also a poset isomorphism.}
\label{eqpbxsmsklJnwjmkcpT}
\end{equation}
Combining \eqref{eqdzhTslRjGnw}, \eqref{eqpbxZsfJklstrbBs}, and \eqref{eqpbxsmsklJnwjmkcpT}, we obtain that $\phi=\psi_2\cdot \psi_1$. Hence, $\phi$ is also a poset isomorphism, proving part \eqref{lemmamainc} of Lemma \ref{lemmamain}.

Finally, if a partial ordering $\leq$ is the reflexive transitive closure of a relation $\rho$, then the covering relation $\prec$ with respect to $\leq$ is obviously a subset of $\rho$. This yields 
part \eqref{lemmamaind} and completes the proof of Lemma~\ref{lemmamain}.
\end{proof}

\section
{Some consequences of Lemma~\ref{lemmamain} and some properties of lamps}
\label{sectioneasycons}
Some statements of this section are explicitly devoted to congruence lattices of slim planar semimodular lattices; they are called corollaries since they are derived from Lemma~\ref{lemmamain}. The rest of the statements of the section deal with lamps.

Convention~\ref{convsRhRtskTskk} raises the (easy) question whether lamps are determined by their foots and how. For an element $u\in L\setminus \set 1$, let $\cov u$ denote the join of all covers of $u$, that is,
\begin{equation}
\cov u:=\bigvee\set{y\in L: u\prec y},\quad \text{provided}\quad u\neq 1.
\label{eqyyZswrRsHnRbp}
\end{equation}
Note  that $\cov 1$ is undefined.
Note also that $\bigvee$ in \eqref{eqyyZswrRsHnRbp} applies actually to one or two joinands  since  each $u\in L\setminus\set 1$ has either a single cover, or it has exactly two covers by Gr\"atzer and Knapp~\cite[Lemma 8]{gratzerknapp1}.
Let $x\in L$ and define the element $\lift x$ by induction on the number of elements of the principal filter $\filter x$ as follows.
\begin{equation}
\lift x=
\begin{cases}
1,&\text{if }x=1;\cr
x^+,&\text{if }\exists y\in\Mir L\text{ such that  }y\prec x^+;\cr
\lift{\cov x},&\text{otherwise}.
\end{cases}
\end{equation}

\begin{lemma}\label{lemmalfT}
If $L$ is a slim rectangular lattice, then $\Peak I=\lift{\foot I}$ holds for every $I\in \Lamps L$.
\end{lemma}

\begin{proof} The proof is trivial by Figure~\ref{figsn} and  \eqref{eqtxtmpFrRHhW}.
\end{proof}

\begin{lemma}[Maximal lamps are boundary lamps]\label{lemmaXcxbndry} If $L$ is a slim rectangular lattice, then the maximal elements of $\Lamps L$ are exactly the boundary lamps.
\end{lemma}

\begin{proof}
We know from Gr\"atzer and Knapp~\cite{gratzerknapp3} that each  slim planar semimodular lattice $L$ with at least three elements has a so-called congruence-preserving rectangular extension $L'$; see also Cz\'edli~\cite{czgrectectdiag} and Gr\"atzer and Schmidt~\cite{ggscht-periodica2014} for stronger versions of this result. Among other properties of this $L'$, we have that  $\Con {L'}\cong\Con L$. 
Hence, to simplify the notation, 
\begin{equation}
\parbox{7cm}{we can assume 
in the proof that $L$ is a slim rectangular lattice with a fixed $\bdia$-diagram.}
\label{eqtxtszHrsrZrknSzkMr}
\end{equation}
Note that the same assumption will be made in many other proofs when we know that $|L|\geq 3$.  
Armed with \eqref{eqtxtszHrsrZrknSzkMr}, 
Lemma~\ref{lemmaXcxbndry} follows trivially from \eqref{eqtxtmpFrRHhW} and Lemma~\ref{lemmamain}. 
\end{proof}

The \emph{set of maximal elements} of a poset $P$ will be denoted by $\Max P$. 

\begin{corollary}[P2 property from Gr\"atzer~\cite{ggSPS8}]\label{corollstzRsG} If $L$ is a slim planar semimodular lattice with at least three elements, then $\Con L$ has at least two coatoms or, equivalently, $\Jir{\Con L}$ has at least two maximal elements.
\end{corollary}

\begin{proof} Assume \eqref{eqtxtszHrsrZrknSzkMr}. The well-known representation theorem of finite distributive lattices  (see, for example, Gr\"atzer~\cite[Theorem 107]{r:Gr-LTFound}) easily implies that 
$\Con L$ has at least two coatoms if an only if $\Jir{\Con L}$ has at least two maximal elements. Hence, Lemma~\ref{lemmaXcxbndry} applies.
\end{proof}

The covering relation in a poset $P$ is denoted by $\prec$ or, if confusion threatens, by $\prec_P$. For sets $A_1$, $A_2$, and $A_3$, the notation $A_3=A_1\mathop{\dot\cup} A_2$ will stand for the conjunction of $A_1\cap A_2=\emptyset$ and $A_3=A_1\cup A_2$.

\begin{corollary}[Bipartite  maximal elements property]\label{corolnWdWr} Let $L$ be a slim planar semimodular lattice with at least three elements and let $D:=\Con L$.  Then there exist nonempty sets  $\lmax{\Jir D}$ and $\rmax{\Jir D}$ such that  
\[\Max{\Jir D}= \lmax{\Jir D} \mathop{\dot\cup} \rmax{\Jir D}
\] 
and for each $x\in \Jir D$ and $y,z\in\Max{\Jir D}$, if $x\jirdprec y$, $x\jirdprec z$, and $y\neq z$, then neither $\set{y,z}\subseteq \lmax{\Jir D}$, nor 
 $\set{y,z}\subseteq \rmax{\Jir D}$. 
Furthermore, when $\Jir D=\Jir{\Con L}$ is represented in the form $\Lamps L$ according to Lemma~\ref{lemmamain}\eqref{lemmamainc}, then $\lmax{\Jir D}$ can be chosen so that its members correspond to the boundary lamps on the top left boundary chain of $L$ while the members of $ \rmax{\Jir D}=\max{\Jir D}\setminus \lmax{\Jir D}$ correspond to the boundary lamps on the top right boundary chain.
\end{corollary}

\begin{proof} Armed with \eqref{eqtxtszHrsrZrknSzkMr} again, we know from Lemma~\ref{lemmaXcxbndry} that  the maximal lamps are on the upper boundary. 
Let $\lmax{\Lamps L}$ and  $\rmax{\Lamps L}$ denote the set of boundary lamps on the top left boundary chain $\filter \cornl L$ and those on the top right boundary chain $\filter \cornr L$, respectively. Since $L$ is rectangular,
none of these two sets is empty. So these two sets form a partition of  $\max{\Lamps L}$.
Let us say that, for $I',J'\in \Lamps L$, 
\begin{equation}
\parbox{7.9cm}{$\Enl {I'}$ and $\Enl {J'}$ are \emph{sufficiently disjoint}
if for every line segment $S$ of positive length in the plane, if $S\subseteq \Enl {I'}\cap \Enl {J'}$, then $S$ is of a normal slope.}
\label{eqpbxsfDjsknkTs}
\end{equation}
Clearly, if $\Enl {I'}$ and $\Enl {J'}$ are sufficiently disjoint, then no nonempty open set of $\mathbb R^2$ is a subset of $\Enl {I'}\cap \Enl {J'}$.

Let $I,J\in \lmax{\Lamps L}$ such that $I\neq J$. 
There are two easy ways to see that $\Enl I$ and $\Enl J$ are sufficiently disjoint: either we apply \eqref{eqpbxZhRsjPzrSt}, or we use \eqref{eqpbxdhzWnhRgbBBsWt} with $\pair i j=\pair 0 k$.
Suppose, for a contradiction, that $K\prec_{\Lamps L} I$ and $K\prec_{\Lamps L} J$. Then, by 
parts \eqref{lemmamaina} and \eqref{lemmamaind} of Lemma~\ref{lemmamain}, $\pair K I\in \rhgeomb$. Similarly, $\pair K J\in \rhgeomb$, and so we have that $\body K\subseteq \Enl I\cap\Enl J$. Since $K$ is an internal lamp, it  contains a precipitous neon tube $S$, which contradicts  the sufficient disjointness of $\Enl I$ and $\Enl J$.  By Lemma~\ref{lemmamain} and left-right symmetry, we conclude Corollary~\ref{corolnWdWr}.
\end{proof}

\begin{corollary}[Dioecious maximal elements property]\label{coroldioec} If $L$ is a slim planar semimodular lattice, $D:=\Con L$,  $x\in \Jir D$, $y\in\Max{\Jir D}$, and $x\jirdprec y$, then there exists an element $z\in\Jir D$ such that $z\neq y$ and $x\jirdprec z$.
\end{corollary}

The adjective ``dioecious'' above is explained by the idea of interpreting $x\prec y$ as ``$x$ is a child of $y$''.

\begin{proof}[Proof of  Corollary~\ref{coroldioec}] 
If $|L|<3$, then $|\Jir L|\leq 1$ an the statement is trivial. Hence, we can assume that $|L|\geq 3$ and that $L$ is rectangular; see \eqref{eqtxtszHrsrZrknSzkMr}.
Assume that $I\in\Lamps L$, $J\in\Max{\Lamps L}$, and $I\prec J$ in $\Lamps L$. We know from Lemma~\ref{lemmaXcxbndry} that $I$ is an internal lamp while $J$ is a boundary lamp. 
For the sake of contradiction, suppose that $J$ is the only cover of $I$ in $\Lamps L$. By left-right symmetry and
Lemma~\ref{lemmaXcxbndry}, we can assume that (the only neon tube of) $J$ is on the top left boundary chain of $L$. 
With reference to \eqref{eqpbxHzwsrZstnC}, the illuminated sets of the lamps of $L_0$, which are boundary lamps, divide the 
full geometric rectangle of $L_0$ into pairwise sufficiently disjoint (topologically closed) rectangles $T_1,\dots, T_m$. That is, $T_1,\dots, T_m$ are the squares (that is, the 4-cells) of the initial grid $L_0$. By \eqref{eqpbxdhzWnhRgbBBsWt}, the illuminated sets of the boundary lamps of $L$ (rather than $L_0$) divide the full geometric rectangle of $L$ into the same rectangles, and the same holds for all $L_i$, $i\in\set{0,1,\dots,k}$. We know from \eqref{eqtxtmpFrRHhW}
that, for some $i\in\set{1,\dots,k}$, each of the four sides of the rectangle $\cirrec I$ is an edge in $L_i$.  Since no two edges of $L_i$ cross each other by planarity (see also Kelly and Rival~\cite[Lemma 1.2]{kellyrival}), it follows from \eqref{eqpbxZtSrptnK} that the sides (in fact, edges) of $\cirrec I$ in $L_i$ do not cross the sides of $T_1$, \dots, $T_m$. Hence, still in $L_i$,
$\cirrec I$ is fully included in one of the $T_1,\dots, T_m$. This also holds in $L$ since $\cirrec I$ is the same in $L$ as in $L_i$ by \eqref{eqtxtmpFrRHhW}.
Hence, there is lamp $K$ on the top right boundary chain of $L$ such that $\cirrec I\subseteq \Enl K$. Hence, $\pair I K\in\rhgeomc$, whence parts \eqref{lemmamaina} and \eqref{lemmamainb} of Lemma~\ref{lemmamain} give that  $I<K$ in $\Lamps L$. By finiteness, 
we can pick a lamp $K'$ such that $I\prec K'\leq K$ in $\Lamps L$. 
We have assumed that $J$ is the only cover of $I$, whereby $K'=J$ and so $J=K'\leq K$. The inequality here cannot be strict since both $J$ and $K$ belong to $\Max{\Lamps L}$. Hence, $J=K$, but this is a contradiction since $J$ is on the top left boundary chain of $L$ while $K$ is on the top right boundary chain. Since $\Jir D\cong \Lamps L$ by Lemma~\ref{lemmamain}, we have proved Corollary~\ref{coroldioec}.
\end{proof}

\begin{corollary}[Two-cover Theorem from  Gr\"atzer~\cite{ggtwocover}]\label{coroltwocover} If $L$ is a slim planar semimodular lattice and $D:=\Con L$, then  for every $x\in \Jir D$, the set $\set{y\in \Jir D: x\jirdprec y}$ of covers of $x$ with respect to $\jirdprec$ consists of at most two elements.
\end{corollary}

\begin{proof} Since the case $|L|<3$ is trivial, we assume \eqref{eqtxtszHrsrZrknSzkMr}.
In virtue of Lemma~\ref{lemmamain}, we can work in $\Lamps L$ rather than in $\Jir D$. 
For lamps $I$ and $J$ of our slim rectangular lattice $L$, we define
\begin{equation}
\begin{aligned}
I\lprec J &\defiff \cirrec I\subseteq \LEnl L\quad\text{ and, similarly,}\cr
I\rprec J &\defiff \cirrec I\subseteq \REnl L.
\end{aligned}
\end{equation}
By parts \eqref{lemmamaina} and \eqref{lemmamaind} of Lemma~\ref{lemmamain}, for any $I,J\in\Lamps L$,
\begin{equation}
\text{if $I\prec J$, then $I\lprec J$ or $I\rprec J$;}
\label{eqtxtsmskntrksmkhsCpl}
\end{equation}
note that $I\lprec J$ and $I\rprec J$ can simultaneously hold.
Based on \eqref{eqtxtmpFrRHhW}, a straightforward induction on $i$ occurring in \eqref{eqpbxHzwsrZstnC} yields that 
\begin{equation}
\parbox{7.6cm}{for each $I\in\Lamps L$, there is at most one $J$ in $\Lamps L$ such that $I\lprec J$. Similarly, $I\rprec K$ holds for at most one $K\in \Lamps L$.}
\label{eqpbxthemsmwknZgKffts}
\end{equation}
Finally, \eqref{eqtxtsmskntrksmkhsCpl} and \eqref{eqpbxthemsmwknZgKffts} imply  Corollary~\ref{coroltwocover}.
\end{proof}

\begin{definition}
For non-horizontal parallel geometric \emph{lines} $T_1$ and $T_2$, we say that $T_1$ is \emph{to the left of} $T_2$ if 
$T_i$ is of the form $\set{\pair{a_i}0+t\cdot\pair{v_x}{v_y}: t\in\mathbb R}$ for $i\in\set{1,2}$ such that $a_1<a_2$. 
Here the vector $\pair{v_x}{v_y}$ is the common direction of $T_1$ and $T_2$ while $(a_i,0)$ is the intersection point of $T_i$ and the $x$-axis. We denote by $T_1\rellambda T_2$ that $T_1$ is left to $T_2$. For parallel \emph{line segments} $S_1$ and $S_2$ of positive lengths, we say that $S_1$ is \emph{to the left of} $S_2$, in notation, $S_1 \rellambda S_2$, if the line containing $S_1$ is to the left of the line containing $S_2$. Let us emphasize that 
if $S_1$ or $S_2$ is of zero length, then $S_1 \rellambda S_2$ fails! 
Next, let $L$ be a slim rectangular lattice, and let $J_0$ and $J_1$ be distinct lamps of $L$. With reference to Definition~\ref{defpPrZwsklsmnszk},  we say that $J_0$ and $J_1$ are \emph{left separatory} lamps if there is a (unique) $i\in\set{0,1}$ such that  
\begin{equation}
\LRoof {J_i} \rellambda \LRoof {J_{1-i}}  \rellambda 
\LFloor {J_i} \rellambda \LFloor {J_{1-i}}.
\label{eqxszRskMshTrP}
\end{equation}

Replacing the left roofs and left floors in \eqref{eqxszRskMshTrP} by right roofs and right floors, respectively, we obtain the concept of \emph{right separatory} lamps. We say that $J_0$ and $J_1$ are \emph{separatory lamps} if $J_0$ and $J_1$ are left separatory or right separatory. 
Finally, if the line segments $\LFloor {J_0}$ and $\LFloor {J_1}$ lie on the same line, then $J_0$ and $J_1$ are \emph{left floor-aligned}. If $\RFloor {J_0}$ and $\RFloor {J_1}$ are segments of the same line, then $J_0$ and $J_1$ are \emph{right floor-aligned}. They are \emph{floor-aligned} if they are left  floor-aligned or right  floor-aligned. 
\end{definition}

\begin{lemma} \label{lemmaseparatoraaligned}
If $I$ and $J$ are distinct lamps of 
a slim rectangular lattice, then these two lamps are neither separatory, nor floor-aligned.
\end{lemma}

\begin{proof} To prove this by contradiction, suppose that the lemma fails. Let $J_0:=I$ and $J_1:=J$.  Using \eqref{eqpbxZtSrptnK}, \eqref{eqpbxwizgRpsrbrWnGskS}, and that $\foot {J_j}$ is in the (topological)  interior of $\cirrec {J_j}$ for $j\in\set{0,1}$, we can find an $i\in\set{0,1}$ such that $\LRoof {J_i}$ or $\LFloor{J_i}$ crosses the upper right edge of $\cirrec{J_{1-i}}$ or left-right symmetrically. This contradicts planarity and completes the proof. Alternatively, we can use an induction on $i$ occurring in \eqref{eqpbxHzwsrZstnC}.
\end{proof}

Since this section is intended to be a ``toolkit'', we formulate the following lemma here; not only its proof but also its complete formulation are left to the next section.

\begin{lemma}\label{lemmadlgzTmsrT}
If $L$ is a slim rectangular lattice, then $\Lamps L=\Lamps {L_k}$ satisfies \eqref{eqpbxkszPrhfWgyxrszrTrl}.
\end{lemma}

\begin{figure}[htb]
\centerline
{\includegraphics[scale=1.1]{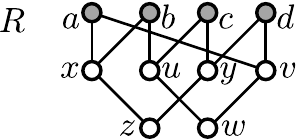}}
\caption{Two-pendant four-crown}\label{figfcrwn}
\end{figure}

\section{Two new properties of congruence lattices of slim planar semimodular lattices}\label{sectionfourcrown}
For posets $X$ and $Y$, we say that $X$ is a \emph{cover-preserving subposet} of $Y$ if $X\subseteq Y$ and, for all $u,v\in X$, 
$u\leq_X v\iff u\leq_Y v$ and $u\prec_X v\iff u\prec_Y v$.
The poset $R$ given in Figure~\ref{figfcrwn} will be called the \emph{two-pendant four-crown}; it is a four-crown decorated with two ``pendants'', $z$ and $w$. 
This section is devoted to the following two properties.

\begin{definition}[Two-pendant four-crown property]\label{deffcRwnpr} We say that a finite distributive lattice $D$  satisfies the \emph{two-pendant four-crown property} if  $R$ given in Figure~\ref{figfcrwn} is not a cover-preserving subposet of $\Jir{D}$ such that the maximal elements of $R$ are maximal in $\Jir{D}$.
\end{definition}

\begin{definition}[Forbidden marriage property]\label{defbDdNmRpzp}
 We say that a finite distributive lattice $D$  satisfies the \emph{forbidden marriage property} if for every $x,y\in \Jir D$ and $z\in\Max{\Jir D}$, if $x\neq y$, $x\jirdprec z$, and $y\jirdprec z$, then there is no $p\in \Jir D$ such that $p\jirdprec x$ and $p\jirdprec y$.
\end{definition}

Now we are in the position to formulate the main theorem of the paper.

\begin{theorem}[Main Theorem]\label{thmmain} If $L$ is a slim planar semimodular lattice, then 
\begin{enumeratei}
\item\label{thmmaina} 
$\Con L$ satisfies the 
the forbidden marriage property, and 
\item\label{thmmainb} 
$\Con L$ satisfies the two-pendant four-crown property;
\end{enumeratei}
see Definitions~\ref{deffcRwnpr} and \ref{defbDdNmRpzp}.
\end{theorem}

\begin{remark}\label{remskzndDWnlc}
The smallest distributive lattice $D$ that fails to satisfy the forbidden marriage property is the eight-element $D_8$ given in Cz\'edli~\cite{czganotesps}. 
Now the result of \cite{czganotesps}, stating that $D_8$ cannot be represented as the congruence lattice of a slim planar semimodular lattice, becomes an immediate consequence of Theorem~\ref{thmmain}. In fact, part \eqref{thmmaina} of Theorem~\ref{thmmain} is a generalization of \cite{czganotesps}.
\end{remark}

\begin{remark}\label{remarksFrzg}
By the well-known representation theorem of finite distributive lattices, see, for example, Gr\"atzer~\cite[Theorem 107]{r:Gr-LTFound}, there is a unique finite distributive lattice $D_R$ such that  $\Jir {D_R}\cong R$. 
By Theorem~\ref{thmmain}\eqref{thmmainb}, there is no slim planar semimodular lattice $L$ such that $\Con L\cong D_R$. Since $D_R$ satisfies the properties mentioned in Corollaries~\ref{corollstzRsG}--\ref{coroltwocover}, so  all the previously known properties, and even the forbidden marriage property, it follows that  part \eqref{thmmainb} of Theorem~\ref{thmmain} is  really a new result.
\end{remark}

\begin{remark}\label{remarkshlmS}
A straightforward calculation shows that  $D_R$ mentioned in Remark~\ref{remarksFrzg} consists of $56$ elements.  Furthermore, we are going to prove that  every finite distributive lattice with less than 56 elements satisfies the two-pendant four-crown property. 
\end{remark}

\begin{figure}[htb]
\centerline
{\includegraphics[scale=1.0]{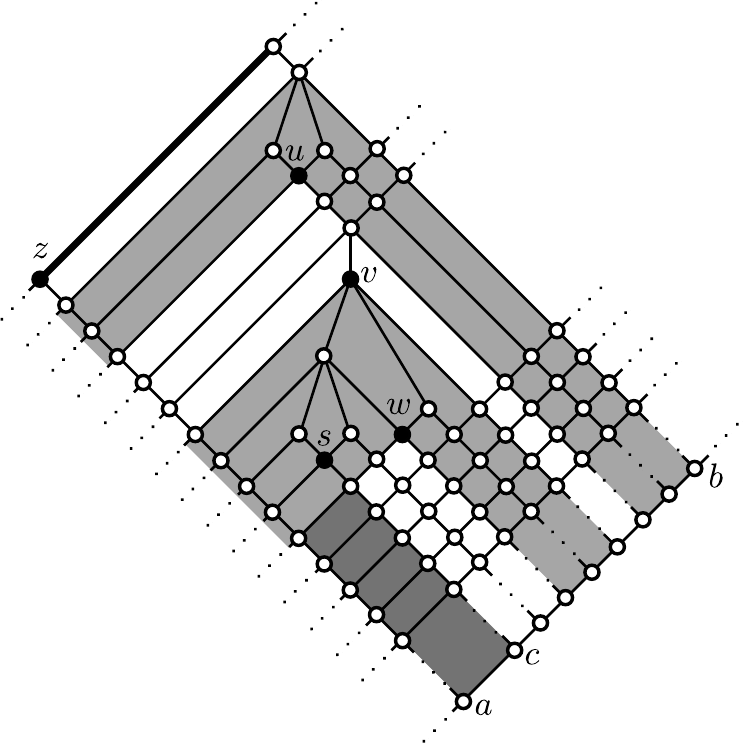}}
\caption{Illustrating the proof of \eqref{eqpbxkszPrhfWgyxrszrTrl}}\label{figfbmr}
\end{figure}

\begin{proof}[Proof of Theorem~\ref{thmmain}] As usual, the case $|L|<3$ is trivial. Hence,  \eqref{eqtxtszHrsrZrknSzkMr} is assumed. Also, we use the notation $D:=\Con L$. By Lemma~\ref{lemmamain}, we can also assume that $\Jir D=\Lamps L$. For $J_0, J_1\in\Lamps L$, we say that
\begin{equation}
\parbox{8cm}{the lamps $J_0$ and $J_1$ are  \emph{independent} if there is a (unique) $i\in\set{0,1}$ such that $\Peak{J_i}\leq \foot {J_{1-i}}$. 
}\label{eqpbxfszKswhvRbg}
\end{equation}

First, we deal with part \eqref{thmmaina}, that is,  with the forbidden marriage property. 
In Figure~\ref{figfbmr}, which is either $L$ or only a part (in fact, an interval) of $L$, there are four internal lamps,
$S$, $U$, $V$, and $W$; their foots are $s$, $u$, $v$, and $w$, respectively. In this figure, for example, $U$ and $V$ are independent but $S$ and $W$ are not. Actually, $\set{S,W}$ is the only two-element subset of $\set{S,U,V,W}$ whose two members are not independent.
It follows from \eqref{eqpbxnfhRmsrlmrtlTslVz}
and Corollary 6.1 of Cz\'edli~\cite{czgrectectdiag}  that, using the terminology of \eqref{eqpbxsfDjsknkTs},
\begin{equation}
\parbox{7cm}{if $J_0$ and $J_1$ are independent lamps, then 
$\Enl{J_0}$ and $\Enl{J_1}$ are sufficiently disjoint.}
\label{eqpbxsfkggdNgn}
\end{equation}
Now assume that $Z$ is a boundary lamp. By left-right symmetry, we can assume that it is on the top left boundary chain; see Figure~\ref{figfbmr} where $z=\foot Z$. The intersection of $\Enl Z=\REnl Z$ with the right boundary chain $\rightb L$ will be denoted by $E(Z)$; it is the (topologically closed) line segment with endpoints $a$ and $b$ in the Figure. Also, 
we define the following set of geometric points
\[
F(Z):=\{q\in E(Z): (\exists U\in \Lamps L)\,(\pair U Z \in\rhgeomc \text{ and }q\in\Enl U)\}.
\]
For example, $F(Z)$ in Figure~\ref{figfbmr} is the line segment with endpoints $b$ and $c$.  If $F(Z)=\emptyset$ or $F(Z)$ is a line segment with its upper endpoint being the same as that of $E(Z)$, then we say that \emph{there is no gap in $F(Z)$}. For example, there is no gap in $F(Z)$ in Figure~\ref{figfbmr}.
With reference to \eqref{eqpbxHzwsrZstnC} and \eqref{eqtxtmpFrRHhW}, we claim that
\begin{equation}
\parbox{6.2cm}{for $i=0,1,\dots,k$, the lower covers of $Z$ in 
$\Lamps{L_i}$ are pairwise independent and, still in $L_i$, there is no gap in $F(Z)$;}
\label{eqpbxkszPrhfWgyxrszrTrl}
\end{equation}
note that our boundary lamp $Z$ is in $L_0$ and so 
\eqref{eqpbxkszPrhfWgyxrszrTrl} makes sense.  
Note also that \eqref{eqpbxkszPrhfWgyxrszrTrl} implies Lemma~\ref{lemmadlgzTmsrT}, because $L=L_k$.

We prove \eqref{eqpbxkszPrhfWgyxrszrTrl} by induction on $i$. The case $i=0$ is trivial since $Z$ has no lower cover in $\Lamps{L_0}$.
In Figure~\ref{figfbmr}, $Z$ has three lower covers: $U$, $V$ and $W$.  (Since $S<W<Z$,  $S$ is not a lower cover.) 
Assume that this figure is the relevant part (that is, $\Enl Z$) of $L_i$ for some $i\in\set{0,1,\dots, k-1}$. Assume also that in $L_{i+1}$, 
$Z$ obtains a new lower cover, $T$. 
We know from \eqref{lemmamaina} and \eqref{lemmamaind} of Lemma~\ref{lemmamain} that $\pair T Z\in\rhgeomc$.
Due to \eqref{eqpbxHzwsrZstnC} and  \eqref{eqtxtmpFrRHhW}, $\cirrec T$ is a distributive $4$-cell in the figure. 
For every lamp $G\in\Lamps{L_i}$ such that $G< Z$ (in particular, if $\pair G Z\in\rhgeomc$), 
$\cirrec T$ cannot be a 4-cell of $\Enl G$ since otherwise $\pair T G\in \rhgeomc$
would lead to $T<G<Z$, contradicting $T\prec Z$. Also, there can be no $G\in \Lamps{L_i}$ such that $\pair G Z\in\rhgeomc$ and $G$ is to the ``south-east'' of $T$, because otherwise $\cirrec T$ would not be distributive. It follows that $T$ is one of the dark-grey cells in the figure, whereby even in $L_{i+1}$, the lower covers of $Z$ remain pairwise independent and
there is no gap in $F(Z)$. Since the figure clearly represents generality, we are done with the induction step from $i$ to $i+1$. This proves  \eqref{eqpbxkszPrhfWgyxrszrTrl}.

Finally, Lemmas~\ref{lemmamain} and \ref{lemmaXcxbndry} translate
part \eqref{thmmaina} of Theorem~\ref{thmmain} to the following statement on $\Lamps L$:
\begin{equation}
\parbox{9.5cm}{if $Z$ is a boundary lamp, $X\prec Z$,  $Y\prec Z$, and $X\neq Y$, then there exists no $P\in \Lamps L$ such that $P\prec X$ and $P\prec Y$.}
\label{eqkmscKpszGTsj}
\end{equation}
To see this, assume the premise. For the sake of contradiction, suppose that there does exist a $P$ described in \eqref{eqkmscKpszGTsj}. By \eqref{lemmamaina} and \eqref{lemmamaind} of  Lemma~\ref{lemmamain}\eqref{lemmamaind}, $\pair P X\in \rhgeomc$ and   $\pair P Y\in \rhgeomc$. Hence, $\cirrec P\subseteq \Enl X\cap\Enl Y$, whereby $\Enl X$ and $\Enl Y$ are not sufficiently disjoint. 
On the other hand, we obtain from  $L=L_k$ and \eqref{eqpbxkszPrhfWgyxrszrTrl} that $\Enl X$ and $\Enl Y$ are independent, whence they are sufficiently disjoint by \eqref{eqpbxsfkggdNgn}. This is a contradiction proving \eqref{eqkmscKpszGTsj} and part \eqref{thmmaina} of Theorem~\ref{thmmain}.

Next, we deal with part \eqref{thmmainb}. 
For the sake of contradiction, suppose that $D=\Con L$
fails to satisfy the two-pendant four-crown property. 
This assumption and Lemma~\ref{lemmamain} yield that  $R$ is a cover-preserving subposet of $\Lamps L=\Jir D$ such that the four maximal elements of $R$ are also maximal in $\Lamps L$; see  Definition~\ref{deffcRwnpr}.
For $a,b,\dots \in R$, the corresponding lamp  will be denoted by $A$, $B$, \dots, that is, by the capitalized version of the notation used in Figure~\ref{figfcrwn}.  
By Lemma~\ref{lemmaXcxbndry}, 
$A,B,C,D$ are boundary lamps. Each of these four lamps is on the top left boundary chain or on the top right boundary chain.

By Corollary~\ref{corolnWdWr},
any two consecutive members of the sequence $A,B,C,D,A$ 
belong to different top boundary chains since they
have a common lower cover. By left-right symmetry, we can assume that $A$ and $C$ are on the top left boundary chain while $B$ and $D$ on the top right one. 
We can assume that $C$ is above $A$ in the sense that $\foot A<\foot C$ 
since otherwise we can relabel $R$ according to the ``rotational'' automorphism that restricts to $\set{a,b,c,d}$ as 
\begin{equation*}
\begin{pmatrix}
a&b&c&d\cr
c&d&a&b
\end{pmatrix}.
\end{equation*}
Also, we can assume that $D$ is above $B$ since otherwise we can extend
\begin{equation*}
\begin{pmatrix}
a&b&c&d\cr
a&d&c&b
\end{pmatrix}.
\end{equation*}
to a ``reflection'' automorphism of $R$ and relabel $R$ accordingly. Note that $A$ and $C$ are not necessarily neighboring boundary lamps, that is, we have $\Peak A\leq \foot C$  but $\Peak A<\foot C$ need not hold. 
Similarly, we only have that $\Peak B\leq \foot D$. The situation is outlined in Figure~\ref{figabcd}. The foots of lamps in the figure are black-filled and any of the two distances marked by curly brackets can be zero. The illuminated sets $\Enl X$ and $\Enl Y$ are dark-grey while $\Enl A$ and $\Enl C$ are (dark and light) grey. Since $X\prec A$ and $X\prec B$, we know from
\eqref{lemmamaina} and \eqref{lemmamaind} of Lemma~\ref{lemmamain} that $\pair  X A\in\rhgeomb$ and $\pair  X B\in\rhgeomb$. Hence, $\body X\subseteq \Enl A\cap \Enl B$,
 in accordance with the figure. Similarly,  $Y\prec C$ and $Y\prec D$ lead to $\body Y\subseteq \Enl C\cap \Enl D$, as it is indicated in Figure~\ref{figabcd}. Note that due to \eqref{defgshscdzrkrrl}, the figure is satisfactorily correct in this aspect.
Hence, the \ashape-shaped $\Enl Y$ is above the \ashape-shaped $\Enl X$. Thus, using \eqref{eqpbxnfhRmsrlmrtlTslVz}, we obtain that 
\begin{equation}
\text{$\Enl X$ and $\Enl Y$ are sufficiently disjoint;}
\label{eqtxtkzslggdKnmsMs}
\end{equation}
see \eqref{eqpbxsfDjsknkTs} for this concept. 
On the other hand, $Z\prec X$ and $Z\prec Y$ together with \eqref{lemmamaina} and \eqref{lemmamaind} of
Lemma~\ref{lemmamain} gives that $\pair Z X\in\rhgeomb$ and 
$\pair Z Y\in\rhgeomb$. Hence, $\body Z\subseteq \Enl X\cap \Enl Y$. Thus, since $Z$ is an internal lamp by Lemma~\ref{lemmaXcxbndry}, $\Enl X\cap \Enl Y$ contains a precipitous neon tube. This contradicts \eqref{eqtxtkzslggdKnmsMs}. 
Note that there is another way to get a contradiction: since 
$\pair Z X\in\rhgeomc$ and $\pair Z Y\in\rhgeomc$, we have that $\cirrec Z \subseteq \Enl X\cap \Enl Y$, which contradicts the fact that $\cirrec Z$ is of positive area (two-dimensional measure) while $\Enl X\cap \Enl Y$ is of area 0. Any of the two contradictions in itself implies part \eqref{thmmainb} and completes the proof of Theorem~\ref{thmmain}.
\end{proof}

\begin{figure}[htb]
\centerline
{\includegraphics[scale=1.0]{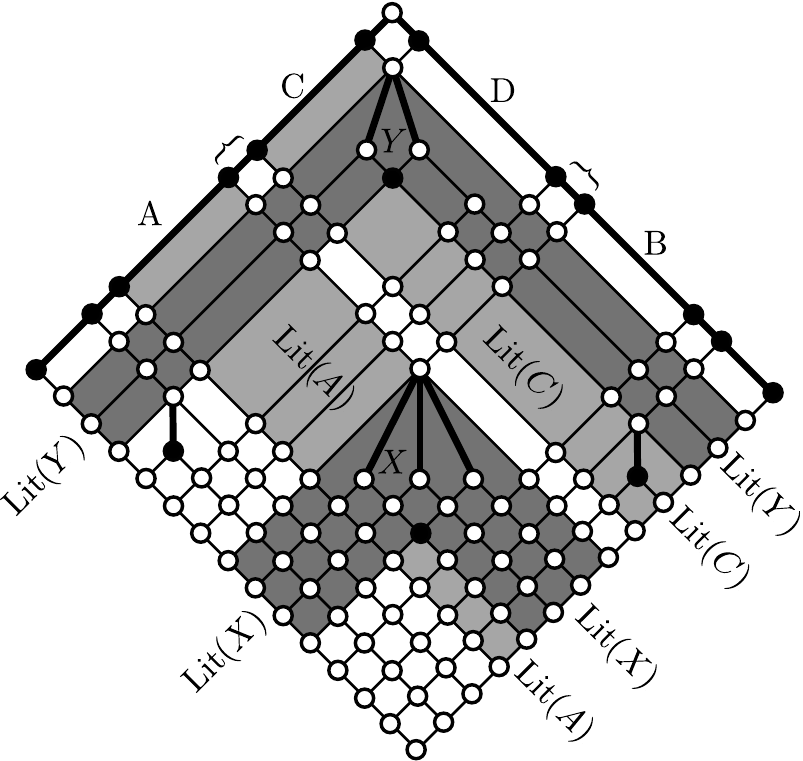}}
\caption{Illustrating the proof of Theorem~\ref{thmmain}\eqref{thmmainb} \label{figabcd}}
\end{figure}

\begin{proof}[Proof of Remark~\ref{remarkshlmS}] 
For the sake of contradiction, suppose that  there exists a distributive lattice $D$ such that $D$ fails to satisfy the two-pendant four-crown property but $|D|<56$. Let $Q:=\Jir D$. By our assumption, $R$ is a subposet of the poset $Q$. Denote by 
$\Her R$ and $\Her Q$ the lattice of down sets (that is, order ideals and the emptyset) of $R$ and $Q$, respectively.  For $X\subseteq Q$ , let $\ideal_Q X:=\{y\in Q: y\leq x$ for some $x\in X\}$. The map $\phi\colon \Her Q\to \Her R$, defined by $X\mapsto X\cap R$ is surjective since, for each $Y\in\Her R$, we have that $\ideal_Q Y\in\Her Q$ and $\phi(\ideal_Q Y)=Y$. Hence, using 
the structure theorem mentioned in Remark~\ref{remarksFrzg}, 
$|D|=|\Her Q|\geq |\Her R|=|D_R|=56$, which is a contradiction proving Remark~\ref{remarkshlmS}.
\end{proof}

We conclude the paper with a last remark.

\begin{remark} Lamps have several geometric properties. 
Many of these properties have already been mentioned, and there are some other properties of technical nature, too.  These properties 
would allow us to represent the congruence lattices of slim planar semimodular lattices in a purely geometric (but quite technical) way. However, this does not seem to be more useful than our technique based on \eqref{eqpbxHzwsrZstnC}--\eqref{eqtxtmpFrRHhW} and the tools presented in the paper. 
\end{remark}

\begin{addedone} Since January 8, 2021, when the first version of the present paper was uploaded to arXiv, the tools developed here have successfully been used in Cz\'edli~\cite{czgaxiombipart} and  Cz\'edli and Gr\"atzer~\cite{czgginprepar}.
\end{addedone}


\begin{thebibliography}{99}


\bibitem{czgreprhomr} 
    Cz\'edli, G.:
    Representing homomorphisms of distributive lattices as
restrictions of congruences of rectangular lattices.
    Algebra Universalis \tbf{67}, 313--345 (2012)


\bibitem{czgtrajcolor} 
    Cz\'edli, G.: Patch extensions and trajectory colorings of slim rectangular lattices.
    Algebra Universalis \tbf{72},  125--154  (2014)

\bibitem{czgcircles}  
    Cz\'edli, G.: 
    Finite convex geometries of circles.
 Discrete Mathematics \tbf{330},  61--75 (2014)

\bibitem{czganotesps} 
    Cz\'edli, G.: 
    A note on congruence lattices of slim semimodular
lattices,
    Algebra Universalis \tbf{72}, 225--230  (2014) 

\bibitem{czgrectectdiag} 
    Cz\'edli, G.: 
    Diagrams and rectangular extensions of planar semimodular lattices.
    Algebra Universalis \tbf{77}, 443--498  (2017)


\bibitem{czgaxiombipart} 
    Cz\'edli, G.: 
    Non-finite axiomatizability of some finite structures.\\
    \texttt{http://arxiv.org/abs/2102.00526}

\bibitem{czgggltsta}  
   Cz\'edli, G.,  Gr\"atzer, G.:
   Planar semimodular lattices and their diagrams. Chapter 3 in: Gr\"atzer, G.,
Wehrung, F. (eds.) Lattice Theory: Special Topics and Applications. Birkh\"auser, Basel (2014)

\bibitem{czgginprepar} 
   Cz\'edli, G.,  Gr\"atzer, G.: A new property of congruence lattices of slim, planar, semimodular lattices.   In preparation

\bibitem{czggghlswing}  
   Cz\'edli, G.,  Gr\"atzer, G., H. Lakser:
   Congruence structure of planar semimodular lattices: the general swing lemma. Algebra Universalis \tbf{79} (2018),  Paper No. 40, 18 pp. 

\bibitem{czgkurusa} 
   Cz\'edli, G., Kurusa, \'A.:
   A convex combinatorial property of compact sets in the plane and its roots in lattice theory. Categories and General Algebraic Structures with Applications 11, 57--92 (2019)\quad
\texttt{http://cgasa.sbu.ac.ir/article\textunderscore{}82639.html}

\bibitem{czgmakay} 
   Cz\'edli, G., Makay, G.:
   Swing lattice game and a direct proof of the swing lemma for planar semimodular lattices.
    Acta Sci. Math. (Szeged) \tbf{83}, 13--29 (2017)


\bibitem{czgschtJH} 
   Cz\'edli, G., Schmidt, E.T.:
   The Jordan-H\"older theorem with uniqueness for groups and semimodular lattices. 
   Algebra Universalis    \tbf{66}, 69--79 (2011)


\bibitem{day}  
  Day, A.: Characterizations of finite lattices that are bounded-homomorphic images or sublattices of free lattices.
  Canad. J. Math.\ \tbf{31}, 69--78 (1979)

\bibitem{fjnbook} 
Freese, R., Je\v zek, J., Nation, J. B.: 
Free lattices. Mathematical Surveys and
Monographs, \tbf{42}, American Mathematical Society, Providence, RI, (1995)

\bibitem{r:Gr-LTFound} 
   Gr\"atzer, G.: 
   Lattice Theory: Foundation.
   Birkh\"auser, Basel (2011)


\bibitem{ggonaresczg} 
    Gr\"atzer, G.:
    On a result of G\'abor Cz\'edli concerning congruence lattices of planar semimodular lattices. 
    Acta Sci. Math. (Szeged) \tbf{81}, 25--32 (2015)

\bibitem{ggconprimperst} 
  Gr\"atzer, G.:
  Congruences and prime-perspectivities in finite lattices. 
  Algebra Universalis \tbf{74},  351--359  (2015)

\bibitem{ggswinglemma}  
 Gr\"atzer, G.: 
 Congruences in slim, planar, semimodular lattices: The Swing Lemma. 
 Acta Sci. Math. (Szeged) \tbf{81}, 381--397 (2015)

\bibitem{ggtwocover} 
    Gr\"atzer, G.:
    Congruences of fork extensions of slim, planar, semimodular lattices.
     Algebra Universalis \tbf{76}, 139--154 (2016)

\bibitem{ggSPS8} 
    Gr\"atzer, G.:
    Notes on planar semimodular lattices. VIII. Congruence lattices of SPS lattices. 
   Algebra Universalis \tbf{81} (2020),  Paper No. 15, 3 pp.


\bibitem{gratzerknapp1} 
   Gr\"atzer, G., Knapp, E.:
   Notes on planar semimodular lattices. I.  Construction. 
   Acta Sci.\ Math.\ (Szeged) \tbf{73}, 445--462 (2007)

\bibitem{gratzerknapp3} 
   Gr\"atzer, G., Knapp, E.:
   Notes on planar semimodular lattices. III. Congruences of rectangular lattices. 
   Acta Sci. Math. (Szeged), \tbf{75}, 29--48 (2009)


\bibitem{gr-nation} 
   Gr\"atzer, G.,  Nation, J.B.:  
    A new look at the Jordan-H\"older theorem for semimodular lattices.
   Algebra Universalis  \tbf{64}, 309--311 (2010)

\bibitem{ggscht-periodica2014}
  Gr\"atzer, G., Schmidt, E.T.:
  An extension theorem for planar semimodular lattices.
  Periodica Math. Hungarica \tbf{69}, 32--40 (2014)


\bibitem{kellyrival} 
  Kelly, D., Rival, I.: 
  Planar lattices. 
  Canad. J. Math. \tbf{27}, 636--665 (1975)


\end{thebibliography}
\end{document}